\documentclass[10pt,reqno]{amsart}
\usepackage{floatflt}
\usepackage{xypic}
\usepackage{rotating}
\usepackage{bbm}
\usepackage{mathrsfs}
\usepackage{diagbox}
\usepackage{dutchcal}
\usepackage{extarrows}
\usepackage{cite}
\usepackage{amsfonts} 
\usepackage[dvipsnames,usenames]{color}
\usepackage{graphicx,latexsym,bm,amsmath,amssymb,verbatim,multicol,lscape}
\makeatletter
\@namedef{subjclassname@2020}{
\textup{2020} Mathematics Subject Classification}
\makeatother
\newtheorem{theorem}{Theorem} [section]

\newtheorem{conjecture}[theorem]{Conjecture}
\newtheorem{lemma}[theorem]{Lemma}

\numberwithin{equation}{section}

\def\bew{\begin{widetext}}
\def\eew{\end{widetext}}
\def\be{\begin{equation}}
\def\ee{\end{equation}}
\def\bea{\begin{eqnarray}}
\def\eea{\end{eqnarray}}

\allowdisplaybreaks
\begin{document}
\title[Factorization of power GCD and power LCM matrices]
{Factorization of power GCD matrices and power LCM matrices
on certain gcd-closed sets}
\begin{abstract}
For integers $x$ and $y$, $(x, y)$ and $[x, y]$ stand
for the greatest common divisor and the least common multiple
of $x$ and $y$ respectively. Denote by $|T|$ the number
of elements of a finite set $T$. Let $a,b$ and $n$ be
positive integers and let $S=\{x_1, \cdots, x_n\}$ be a set of $n$
distinct positive integers. We denote by $(S^a)$ (resp. $[S^a]$)
the $n\times n$ matrix having the $a$th power of $(x_i,x_j)$
(resp. $[x_i,x_j]$) as its $(i,j)$-entry. For any $x\in S$, define
$G_{S}(x):=\{d\in S: d<x, d|x \ {\rm and} \ (d|y|x, y\in S)
\Rightarrow y\in \{d,x\}\}$. In this paper, we show that if $a|b$
and $S$ is gcd closed (namely, $(x_i, x_j)\in S$ for all integers
$i$ and $j$ with $1\le i, j\le n$) and $\max_{x\in S}\{|G_S (x)|\}=3$
such that any elements $y_1,y_2\in G_S(x)$ satisfy that $[y_1,y_2]=x$
and $(y_1,y_2)\in G_S(y_1)\cap G_S(y_2)$), then $(S^a)|(S^b)$, $(S^a)|[S^b]$
and $[S^a]|[S^b]$ hold in the ring $M_{n}({\mathbb Z})$. This extends
the Chen-Hong-Zhao theorem gotten in 2022. This also partially confirms
a conjecture of Hong raised in [S.F. Hong, Divisibility among power
GCD matrices and power LCM matrices, {\it Bull. Aust. Math. Soc.},
doi:10.1017/S0004972725100361].
\end{abstract}
\author[G.Y. Zhu]{Guangyan Zhu$^*$}
\address{School of Mathematics and Statistics, Hubei Minzu
University, Enshi 445000, P.R. China}
\email{2009043@hbmzu.edu.cn}
\author[Y.Y. Luo]{Yuanyuan Luo}
\address{College of Mathematics and Statistics, Guizhou University of Finance and Economics,
Guiyang 550025, P.R. China}
\email{yuanyuanluoluo@163.com}
\author[J.X. Wan]{Jixiang Wan}
\address{College of Mathematics and Physics, Mianyang Teachers' College,
Mianyang, 621000, P.R. China}
\email{xiangxiangwanli@163.com; 2873555@qq.com}
\thanks{$^*$ G.Y. Zhu is the corresponding author.
This work was supported in part by the Startup Research Fund of Hubei Minzu University for Doctoral Scholars (Grant No. BS25008),
in part by the National Science Foundation of China (Grant No. 12161012),
and in part by the Scientific Research Foundation for the Introduction of Talent in GUFE (Grant No. 2018YJ85).}
\maketitle

\section{Introduction}
Let $(x, y)$ and $[x, y]$ denote the greatest common divisor
and the least common multiple of the integers $x$ and $y$,
respectively. Let $a, b$ and $n$ be positive integers.
Let $|T|$ stand for the cardinality of a finite set $T$ of integers.
Let $f$ be an arithmetic function and let $S=\{x_1, \cdots, x_n\}$ be a
set of $n$ distinct positive integers. Let $(f(S))$ and $(f[S])$
stand for the $n\times n$ matrices having $f((x_i, x_j))$
and $(f[x_i, x_j])$ as its $(i,j)$-entry, respectively.
A set $S$ is called {\it factor closed} if the conditions
$x\in S$ and $d|x$ together with $d>0$ imply that $d\in S$.
We say that $S$ is {\it gcd closed} if $(x_i, x_j)\in S$
for all $i$ and $j$ with $1\le i,j\le n$. Evidently,
any factor closed set is gcd closed but not conversely. In 1875, Smith
\cite{[S-PLMS1875]} published his famous theorem stating that
if $S$ is factor closed, then $\det(f(x_i, x_j))=\prod_{k=1}^n(f*\mu)(x_k)$,
where $f*\mu$ is the {\it Dirichlet convolution} of $f$
and the M\"{o}bius function $\mu$
and is defined for any positive integer $x$ by
$(f*\mu)(x):=\sum_{d|x}f(d)\mu\big(\frac{x}{d}\big).$
Since then lots of generalizations of Smith's determinant
and related results had published (see, for example,
\cite{[AYK-LAA17]}-\cite{[LH-BAMS18]} and \cite{[TL-LAA13]}-
\cite{[ZLX-AIMS2022]}). For any positive integer $x$,
let $\xi_a$ be the arithmetic function defined by $\xi_a(x):=x^a$.
The $n\times n$ matrix $(\xi_a(x_i, x_j))$ (abbreviated
by $(S^a)$) and $(\xi_a[x_i, x_j])$ (abbreviated by
$[S^a]$) are called {\it $a$th power GCD matrix} on
$S$ and {\it $a$th power LCM matrix} on $S$, respectively.
In 1993, Bourque and Ligh \cite{[BL-LMA1993]} extended the
Beslin-Ligh result \cite{[BL-BAMS89]} and Smith's
determinant by proving that if $S$ is gcd closed,
then $\det(S^a)=\prod_{k=1} ^n \alpha_{\xi_a}(x_k)$, where
\begin{align}\label{eq1.1}
\alpha_{\xi_a}(x_k):=\sum_{d|x_k\atop d\nmid x_t,x_t<x_k}
(\xi_a*\mu)(d).
\end{align}
An integer $x\ge 1$ is called {\it square free} if $x$
is not divisible by the square of any prime number.
In 2018, Lin and Hong \cite{[LH-BAMS18]} generalized
Smith's theorem \cite{[S-PLMS1875]} and the Hong-Hu-Lin
theorem \cite{[HHL-AMH16]} by showing that for any integer
$t$ with $1\le t\le n$, if $S$ is factor closed, then
$$\text {det}{(f(x_i, x_j))}_{1\le i,j\le n\atop i\neq t,j\neq t}
=\overset{n}{\underset{l=1\atop x_t|x_l,
\frac{x_l}{x_t} \text {square free}}{\sum}}
\prod\limits_{k=1\atop k\neq l}^n\left(f*\mu\right)(x_k).$$
Besides, the eigen structures of power GCD and power LCM
matrices were explored by Wintner \cite{W}, Bourque and Ligh
\cite{[BL-JNT93]}, Hong and Loewy \cite{[HLo-GMJ04],[HLo-IJNT11]},
Hong and Lee \cite{{[HLe-GMJ08]}}, Altinisik \cite{[A-LAA09]},
Mattila and Haukkanen \cite{[MH]} and Kaarnioja \cite{Ka-JCTA21}.

Let $x,y \in S$ with $x<y$. If $x|y$ and the conditions
$x|d|y$ and $d\in S$ imply that $d\in \{x,y\}$, then
we say that $x$ is a {\it greatest-type divisor} of
$y$ in $S$. For $x\in S$, $G_S(x)$ stands for the set
of all greatest-type divisors of $x$ in $S$. The concept
of greatest-type divisor was introduced by Hong and
played a key role in his solution \cite{[H-JA99]} of
the Bourque-Ligh conjecture \cite{[BL-LAA92]}. Bourque
and Ligh \cite{[BL-LAA95]} showed that
if $S$ is factor closed and $a\ge 1$ is an integer, then the $a$th
power GCD matrix $(S^a)$ divides the $a$th power LCM matrix
$[S^a]$ in the ring $M_n(\mathbb Z)$ of $n\times n$ matrices
over the integers. That is, there exists an $A\in M_n(\mathbb Z)$
such that $[S^a]=(S^a)A$ or $[S^a]=A(S^a)$. Hong \cite{[H-LAA02]}
showed that such a factorization is no longer true in general
if $S$ is gcd closed and $\max_{x \in S}\{|{G_S(x)}|\}=2$.
Korkee and Haukkanen \cite{[KH-LAA08]} and Chen et al.
\cite{[CFHQ-PMD22]} extended the result of Bourque and Ligh
and that of Hong. A characterization on the gcd-closed set $S$
with $\max_{x\in S}\{|G_S(x)|\}=2$ (resp. $\max_{x\in S}\{|G_S(x)|\}=3$)
such that $(S^a)|[S^a]$ holds in the ring $M_{|S|}(\mathbb Z)$ was
given in \cite{[FHZ-DM2009]} (resp. \cite{[ZCH-JCTA2022]}).

On the other hand, Hong \cite{[H-LAA08]} initially investigated
the divisibility properties among power GCD matrices and among
power LCM matrices. It was proved in \cite{[H-LAA08]} that
$(S^a)|(S^b), (S^a)|[S^b]$ and $[S^a]|[S^b]$ hold in the ring
$M_{n}(\mathbb Z)$ if $a|b$ and $S$ is a divisor chain (that is,
$x_{\sigma(1)}|\cdots|x_{\sigma(n)}$ for a permutation $\sigma$ of
$\{1,\cdots, n\}$), and such factorizations are no longer true if
$a\nmid b$ and $|S|\ge 2$. Evidently, a divisor chain is gcd
closed but the converse is not true. Recently, Hong \cite{[H-BAMS25]}
established the same divisibility result when $S$ is factor closed.
As in \cite{[H-BAMS25]}, for any set $S$ of positive integers and for
any $x\in S$ with $|G_S(x)|\ge 2$, we say that the two distinct greatest-type
divisors $y_1$ and $y_2$ of $x$ in $S$ {\it satisfy the condition
$\mathcal{G}$} if $[y_1,y_2]=x$ and $(y_1,y_2)\in G_S(y_1)\cap G_S(y_2)$.
We say that $x$ {\it satisfies the condition $\mathcal{G}$} if any two
distinct greatest-type divisors of $x$ in $S$ satisfy the condition
$\mathcal{G}$. Moreover, we say that a set $S$ of positive integers
{\it satisfies the condition $\mathcal{G}$} if any element $x$ in $S$
satisfies that either $|G_S(x)|\le 1$, or $|G_S(x)|\ge 2$ and $x$ satisfies
the condition $\mathcal{G}$. In \cite{[H-BAMS25]}, Hong showed if
$a|b$ and $S$ is a factor closed set, then $(S^a)|(S^b), (S^a)|[S^b]$ and
$[S^a]|[S^b]$ hold in the ring $M_{n}(\mathbb Z)$. Meanwhile, Hong
\cite{[H-BAMS25]} proved that any factor-closed set is a gcd-closed
set satisfying the condition $\mathcal{G}$. Furthermore, Hong proposed
the following conjecture.

\begin{conjecture}{\rm {\cite[Conjecture 3.4]{[H-BAMS25]}}} \label{p1.1}
Let $a$ and $b$ be positive integers with $a|b$
and let $S$ be a gcd-closed set satisfying the condition $\mathcal G$. Then
in the ring $M_{|S|}(\mathbb Z)$, we have $(S^a)|(S^b)$, $(S^a)|[S^b]$ and
$[S^a]|[S^b]$.
\end{conjecture}
By Zhu's theorem \cite{[Z-IJNT22]} and the Zhu-Li result \cite{[ZL-BAMS22]},
one knows that this conjecture is true when $a|b$ and $S$ is a gcd-closed set
with $\max_{x\in S}\{|G_S (x)|\}=1$. Also Conjecture \ref{p1.1} is proved by
Wan and Zhu \cite{[WZ]} when $a|b$ and $S$ is a gcd-closed set satisfying the
condition $\mathcal{G}$ and $\max_{x\in S}\{|G_S (x)|\}=2$. But Conjecture
\ref{p1.1} is still kept open when $S$ is a gcd-closed set satisfying the
condition $\mathcal{G}$ and $\max_{x\in S}\{|G_S (x)|\}\ge 3$. One remarks
that Wan and Zhu \cite{[WZ]} showed the existences of gcd-closed sets $S$
with $\max_{x\in S}\{|G_{S}(x)|\}=2$ and the condition $\mathcal{G}$ not
being satisfied and infinitely many integers $b\ge 2$ such that $(S)\mid (S^b)$
(resp. $(S)\mid [S^b]$ and $[S]\mid [S^b]$).

In this paper, our main goal is to explore Conjecture \ref{p1.1} for the case
$\max_{x\in S} \{|G_S(x)|\}=3$. Actually, we will show that Conjecture \ref{p1.1}
is true if $a|b$ and $S$ is a gcd-closed set with $\max_{x\in S}\{|G_S(x)|\}=3$
and the condition $\mathcal{G}$ being satisfied. That is, we have the following
main results of this paper.

\begin{theorem}\label{Theorem 1.2}
Let $S$ be a gcd-closed set satisfying the condition $\mathcal G$ and $\max_{x\in S}
\{|G_S(x)|\}=3$ and let $a$ and $b$ be positive integers with $a|b$. Then the $a$-th
power GCD matrix $(S^a)$ divides each of the $b$-th power GCD matrix $(S^b)$ and the
$b$-th power LCM matrix $[S^b]$ in the ring $M_{|S|}({\mathbb Z})$.
\end{theorem}

\begin{theorem}\label{Theorem 1.3}
Let $S$ be a gcd-closed set satisfying the condition $\mathcal G$ and $\max_{x\in S}
\{|G_S(x)|\}=3$ and let $a$ and $b$ be positive integers with $a|b$.
Then the $a$-th power LCM matrix $[S^a]$ divides the $b$-th power LCM
matrix $[S^b]$ in the ring $M_{|S|}({\mathbb Z})$.
\end{theorem}

We define $S_\sigma:=\{x_{\sigma (1)},\cdots, x_{\sigma (n)}\}$
for any permutation $\sigma$ on the set $\{1, \cdots, n \}$.
It is easy to check that for any permutation $\sigma$
on $\{1, \cdots, n \}$, a necessary and sufficient condition
for $(S^a)|(S^b)$ to hold in the ring $M_n({\mathbb Z})$
is $(S_{\sigma}^a)|(S_{\sigma}^b)$ being true in the
ring $M_n({\mathbb Z})$, and a necessary and sufficient
condition for $(S^a)|[S^b]$ to hold in the ring
$M_n({\mathbb Z})$ is $(S_{\sigma}^a)|[S_{\sigma}^b]$
being true in the ring $M_n({\mathbb Z})$. So, without
loss of any generality, we always assume that the set
$S=\{x_1,\cdots, x_n \}$ satisfies that $x_1<\cdots <x_n$.

This paper is organized as follows. We present in Section 2
several preliminary lemmas that are needed in the proofs of
our main results. Subsequently, we give the proof of Theorem
\ref{Theorem 1.2} in Section 3. Finally, Section 4 is devoted to
the proof of Theorem \ref{Theorem 1.3}.

\section{Preliminary lemmas}
In this section, we present some lemmas that are needed in
the proofs of Theorems \ref{Theorem 1.2} and \ref{Theorem 1.3}.
We begin with a known result due to Bourque and Ligh
\cite{[BL-LMA1993]}.

\begin{lemma} {\rm \cite[Theorem 3]{[BL-LMA1993]}} \label{Lemma 2.1}
If $S$ is gcd closed and $(f(S))$ is
nonsingular, then for any integers $i$ and
$j$ with $1\le i, j\le n$, we have
$$((f(S))^{-1})_{ij}:={\underset{x_i|x_k\atop x_j|x_k}
{\sum}}\frac{c_{ik}c_{jk}}{\alpha_f(x_k)}$$
with
\begin{align*}
\alpha_f(x_k):={\underset{d|x_k\atop d\nmid x_t, x_t<x_k}{\sum}}(f*\mu)(d)
\end{align*}
and
\begin{align}\label{eq2.1}
c_{ij}:=\sum _{dx_i|x_j\atop dx_i\nmid x_t, x_t<x_j}\mu(d).
\end{align}
\end{lemma}

\begin{lemma}\label{Lemma 2.2}
If $S$ is gcd closed, then the power GCD matrix $(S^a)$
is nonsingular and for arbitrary integers $i$ and
$j$ with $1\le i, j\le n$, one has
\begin{align*}
((S^a)^{-1})_{ij}:=\sum_{x_i|x_k\atop x_j|x_k}
\frac{c_{i k}c_{j k}}{\alpha _{\xi_a}(x_k)}
\end{align*}
with $c_{ij}$ being defined as in {\rm(\ref{eq2.1})}
and $\alpha_{\xi_a}(x_k)$ being defined as in {\rm(\ref{eq1.1})}.
\end{lemma}
\begin{proof} By \cite[Example 1 (ii)]{[BL-LMA1993]},
one can deduce that the power GCD matrix $(S^a)$
is positive definite, and so is nonsingular. Then
Lemma \ref{Lemma 2.1} applied to $f=\xi_a$ gives us
the expected result. Lemma \ref{Lemma 2.2} is proved.
\end{proof}

For any positive integer $x$, $\frac{1}{\xi_a}$ is
the arithmetic function defined by $\frac{1}{\xi_a}(x)=\frac{1}{x^a}$.

\begin{lemma}{\rm\cite[Lemma 2.1]{[H-AA2004]}}\label{Lemma 2.4}
If $S$ is gcd closed, then
\begin{align}\label{eq2.2}
\det[S^a]=\prod\limits_{k=1}^n x_k^{2a}\alpha_{\frac{1}{\xi_a}}(x_k),
\end{align}
where
\begin{align}\label{eq2.3}
\alpha_{\frac{1}{\xi_a}}(x_k)={\underset{{d|x_k\atop d\nmid x_t,x_t<x_k}}{\sum}}\Big(\frac{1}{\xi_a}*\mu\Big)(d).
\end{align}
\end{lemma}

\begin{lemma}{\rm{\cite{[H-AA2004]}}} \label{Lemma 2.5}
Let $S$ be a gcd-closed set and let $G_S(x_k)=\{y_{k1},\cdots,y_{k,l_k}\}$
be the set of the greatest-type divisors of $x_k(1\le k\le n)$ in $S$.
Then
$$
\alpha_{\xi_a}(x_k)=x_k^a+\sum_{t=1}^{l_k}(-1)^t\sum\limits_{1\le i_1<\cdots<i_t\le l_k}(x_k,y_{k,i_1},\cdots,y_{k,i_t})^a
$$
and
$$
\alpha_{\frac{1}{\xi_a}}(x_k)=x_k^{-a}+\sum_{t=1}^{l_k}(-1)^t\sum
\limits_{1\le i_1<\cdots<i_t\le l_k}(x_k,y_{k,i_1},\cdots,y_{k,i_t})^{-a}
$$
with $\alpha_{\xi_a}(x_k)$ being determined as in \rm(\ref{eq1.1}) and
$\alpha_{\frac{1}{\xi_a}}(x_k)$ being determined as in \rm(\ref{eq2.3}).
\end{lemma}

\begin{lemma}{\rm\cite[Lemma 2.3]{[FHZ-DM2009]}}
\label{Lemma 2.6}
Let S be a gcd-closed set of $n\ge 2$ distinct
positive integers and let $c_{ij}$ be defined
as in {\rm(\ref{eq2.1})}. Then
$$
c_{r1}=\left\{\begin{aligned}
{1}& \quad if\ r=1,\\
{0}&\quad otherwise.
\end{aligned}
\right.
$$
If $2\le m\le n$ and $G_S(x_m)=\{x_{m_0}\}$, then
$$
c_{rm}=\left\{
\begin{aligned}
{-1}&\quad if\ r=m_0,\\
{1}&\quad if\ r=m,\\
{0}&\quad otherwise.
\end{aligned}
\right.
$$
If $3\le m\le n$ and $G_S(x_m)=\{x_{m_{01}},\ x_{m_{02}}\}$
and $x_{m_{03}}=(x_{m_{01}},\ x_{m_{02}})$, then
$$
c_{rm}=\left\{
\begin{aligned}
{-1}&\quad if \ r=m_{01}\ or\ r=m_{02},\\
{1}&\quad if\ r=m\ or\ m_{03},\\
{0}&\quad otherwise.
\end{aligned}
\right.
$$
\end{lemma}

\begin{lemma}\label{Lemma 2.7}
Let $S$ be a gcd-closed set satisfying the condition
$\mathcal G$. Let $x_m \in S$ with $G_s(x_m)=\{ x_{m_1},x_{m_2},x_{m_3}\}$,
$x_{m_4}=(x_{m_1},x_{m_2},x_{m_3})$ and
$x_{m_{ij}}=(x_{m_i},x_{m_j})$ for $1\le i<j\le 3$. Then
\begin{align*}
c_{rm}=\left\{\begin{array}{cl}
1, &r=m\ \hbox{or}\ r=m_{ij}(1\le i<j\le 3),\\
-1, &r=m_i(1\le i\le 4),\\
0, &\hbox{otherwise}.
\end{array}\right.
\end{align*}
\end{lemma}
\begin{proof}
From the definition of $c_{ij}$ as in $(\ref{eq2.1})$,
if $x_r \nmid x_m$, then we have $c_{rm} = 0$.
In what follows we let $x_r \mid x_m$. First, let $r=m$.
Then we can easily deduce that
\begin{align*}
c_{mm}=\sum\limits_{dx_m \mid x_m \atop dx_m
\nmid x_t, x_t< x_m}\mu(d) = \mu(1) = 1.
\end{align*}

Let $x_r=x_{m_i}(i=1,2,3)$. We can compute and get that
\begin{align*}
c_{m_im}=\sum\limits_{dx_{m_i} \mid x_m \atop dx_{m_i}
\nmid x_t, x_t< x_m}\mu(d)=\sum\limits_{d \mid \frac{x_m}{x_{m_i}} \atop d\nmid \frac{x_t}{x_{m_i}}, x_t< x_m}\mu(d)
=\sum\limits_{d \mid \frac{x_m}{x_{m_i}}}\mu(d)-\sum\limits_{d \mid \frac{x_{m_i}}{x_{m_i}}}\mu(d)=0-1=-1.
\end{align*}
			
Let $x_r=x_{m_{ij}}(1\le i<j\le 3)$. We have
\begin{align*}
c_{m_{ij}m}=&\sum\limits_{dx_{m_{ij}} \mid x_m \atop dx_{m_{ij}}
\nmid x_t, x_t< x_m}\mu(d)\\
=&\sum\limits_{d \mid \frac{x_m}{x_{m_{ij}}} \atop d\nmid \frac{x_t}{x_{m_{ij}}}, x_t< x_m}\mu(d)\\
=&\sum\limits_{d \mid \frac{x_m}{x_{m_{ij}}}}\mu(d)-\sum\limits_{d \mid \frac{x_{m_i}}{x_{m_{ij}}}}\mu(d)
-\sum\limits_{d \mid \frac{x_{m_j}}{x_{m_{ij}}}}\mu(d)+\sum\limits_{d \mid \frac{x_{m_{ij}}}{x_{m_{ij}}}}\mu(d)\\
=&0-0-0+1
=1.
\end{align*}

Let $x_r=x_{m_4}$. One has
\begin{align*}
c_{m_4m}=&\sum\limits_{dx_{m_4}\mid x_m \atop dx_{m_4}
\nmid x_t, x_t< x_m}\mu(d)\\
=&\sum\limits_{d \mid \frac{x_m}{x_{m_4}} \atop d\nmid \frac{x_t}{x_{m_4}}, x_t< x_m}\mu(d)\\
=&\sum\limits_{d \mid \frac{x_m}{x_{m_4}}}\mu(d)-\sum\limits_{d \mid \frac{x_{m_1}}{x_{m_4}}}\mu(d)
-\sum\limits_{d \mid \frac{x_{m_2}}{x_{m_4}}}\mu(d)-\sum\limits_{d \mid \frac{x_{m_3}}{x_{m_4}}}\mu(d)\\
&+\sum\limits_{d \mid \frac{x_{m_{12}}}{x_{m_4}}}\mu(d)+\sum\limits_{d \mid \frac{x_{m_{13}}}{x_{m_4}}}\mu(d)
+\sum\limits_{d \mid \frac{x_{m_{23}}}{x_{m_4}}}\mu(d)-\sum\limits_{d \mid \frac{x_{m_4}}{x_{m_4}}}\mu(d)\\
=&0-0-0-0+0+0+0-1
=-1.
\end{align*}

Now we treat with the case that $x_r|x_m$ with $r \neq m, m_i(i=1,2,3),m_{ij}(1\le i<j\le 3),m_4$.
Since $S$ is a gcd-closed set satisfying the condition $\mathcal G$, $x_m \in S$ and $|G_s(x_m)|=3$,
At this moment, without loss of generality,  we only need to consider the following three cases:

{\sc Case 1}. $x_r|x_{m_1}$, $x_r|x_{m_2}$ and $x_r|x_{m_3}$. This means $x_r|x_{m_4}$.
Note that $x_r\ne x_{m_4}$. So $\frac{x_{m_4}}{x_r}\ge2$. Thus
\begin{align*}
c_{rm}=&\sum\limits_{dx_r|x_m \atop dx_r
\nmid x_t, x_t< x_m}\mu(d)\\
=&\sum\limits_{d \mid \frac{x_m}{x_r} \atop d\nmid \frac{x_t}{x_r}, x_t< x_m}\mu(d)\\
=&\sum\limits_{d \mid \frac{x_m}{x_r}}\mu(d)-\sum\limits_{d \mid \frac{x_{m_1}}{x_r}}\mu(d)
-\sum\limits_{d \mid \frac{x_{m_2}}{x_r}}\mu(d)-\sum\limits_{d \mid \frac{x_{m_3}}{x_r}}\mu(d)\\
&+\sum\limits_{d \mid \frac{x_{m_{12}}}{x_r}}\mu(d)+\sum\limits_{d \mid \frac{x_{m_{13}}}{x_r}}\mu(d)
+\sum\limits_{d \mid \frac{x_{m_{23}}}{x_r}}\mu(d)-\sum\limits_{d \mid \frac{x_{m_4}}{x_r}}\mu(d)\\
=&0-0-0-0+0+0+0-0
=0.
\end{align*}

{\sc Case 2}. $x_r|x_{m_1}$, $x_r|x_{m_2}$ and $x_r\nmid x_{m_3}$. We can deduce that $x_r|x_{m_{12}}$.
Notice that $\frac{x_{m_{12}}}{x_r}\ge 2$. Therefore
\begin{align*}
c_{rm}=&\sum\limits_{dx_r|x_m \atop dx_r
\nmid x_t, x_t< x_m}\mu(d)\\
=&\sum\limits_{d \mid \frac{x_m}{x_r} \atop d\nmid \frac{x_t}{x_r}, x_t< x_m}\mu(d)\\
=&\sum\limits_{d \mid \frac{x_m}{x_r}}\mu(d)-\sum\limits_{d \mid \frac{x_{m_1}}{x_r}}\mu(d)
-\sum\limits_{d \mid \frac{x_{m_2}}{x_r}}\mu(d)+\sum\limits_{d \mid \frac{x_{m_{12}}}{x_r}}\mu(d)\\
=&0-0-0+0
=0.
\end{align*}

{\sc Case 3}. $x_r|x_{m_1}$, $x_r\nmid x_{m_2}$ and $x_r\nmid x_{m_3}$. Notice that $\frac{x_{m_1}}{x_r}\ge 2$.
Hence
\begin{align*}
c_{rm}=&\sum\limits_{dx_r|x_m \atop dx_r
\nmid x_t, x_t< x_m}\mu(d)=\sum\limits_{d \mid \frac{x_m}{x_r} \atop d\nmid \frac{x_t}{x_r}, x_t< x_m}\mu(d)
=\sum\limits_{d \mid \frac{x_m}{x_r}}\mu(d)-\sum\limits_{d \mid \frac{x_{m_1}}{x_r}}\mu(d)
=0-0=0.
\end{align*}

This finished the proof of Lemma \ref{Lemma 2.7}.
\end{proof}

\begin{lemma}{\rm\cite{[H-AC06]}}\label{Lemma 2.8}
Let $S$ be gcd closed such that $\max_{x\in S}\{|G_S(x)|\}\le 3$
and $|S|=n$. Let $\alpha_{\frac{1}{\xi_a}}(x_k)$ be defined as in {\rm (\ref{eq2.3})}.
Then $\alpha_{\frac{1}{\xi_a}}(x_k)\neq0$ for any integer $k$ with $1\le k\le n$.
\end{lemma}

\begin{lemma}\label{Lemma 2.9}
Let $S$ be a gcd-closed set satisfying $\max_{x\in S}
\{|G_S(x)|\}\le 3$. Then the $a$-th power LCM matrix $[S^a]$
is nonsingular and for all integers $i$ and $j$ with
$1\le i, j\le n$, one has
$$([S^a]^{-1})_{ij}:=\frac{1}{x_i^a x_j^a}{\underset{x_i|x_k\atop x_j|x_k}{\sum}}\frac{c_{ik}c_{jk}}{\alpha_{\frac{1}{\xi_a}}(x_k)}$$
with $c_{ij}$
being defined as in {\rm (\ref{eq2.1})} and $\alpha_{\frac{1}{\xi_a}}(x_k)$
being defined as in {\rm (\ref{eq2.3})}.
\end{lemma}
\begin{proof}
Since ${[x_i,x_j]}^a=x_i^ax_j^a/{(x_i,x_j)}^a$, we have
\begin{align}\label{eq2.5}
[S^a]={\rm diag}(x_1^a,\cdots,x_n^a)\cdot
\Big(\frac{1}{\xi_a}(x_i,x_j)\Big)\cdot{\rm diag}(x_1^a,\cdots,x_n^a).
\end{align}
It then follows that
\begin{align}\label{eq2.6}
\det[S^a]=\det\big(\frac{1}{\xi_a}(x_i,x_j)\big)\cdot
\prod\limits_{k=1}^nx_k^{2a}.
\end{align}
Then from (\ref{eq2.2}) and (\ref{eq2.6}), we can derive that $$\det\big(\frac{1}{\xi_a}(x_i,x_j)\big)
=\prod\limits_{k=1}^n\alpha_{\frac{1}{\xi_a}}(x_k).$$
Lemma \ref{Lemma 2.8} tells us that $\alpha_{\frac{1}{\xi_a}}(x_k)\neq0$
for all integers $k$ with $1\le k\le n$.
So the matrix $\big(\frac{1}{\xi_a}(x_i,x_j)\big)$
is nonsingular.

With Lemma \ref{Lemma 2.1} applied to $f=\frac{1}{\xi_a}$,
one gets that
\begin{align}\label{eq2.7}
\Big(\Big(\frac{1}{\xi_a}(x_i,x_j)\Big)^{-1}\Big)_{ij}
={\underset{x_i|x_k\atop x_j|x_k}
{\sum}}\frac{c_{ik}c_{jk}}{\alpha_{\frac{1}{\xi_a}}(x_k)}.
\end{align}
Therefore the required result follows immediately
from (\ref{eq2.5}) and (\ref{eq2.7}). So Lemma
\ref{Lemma 2.9} is proved.
\end{proof}

Now define the set $A_S(z,x):=\{u\in S: z|u|x, u\ne z\}$.

\begin{lemma}{\rm\cite{[ZCH-JCTA2022]}}\label{Lemma 2.10}
Let $S$ be a gcd-closed set and let $x\in S$ satisfy
$|G_S(x)|\ge 2$, $y\in G_S(x)$. Let $z\in S$ be such that
$z|x, z\neq x$ and $z\nmid y$. If $A_S(z,x)$ satisfies
the condition $\mathcal G$, then $[y,z]=x$.
\end{lemma}


\begin{lemma}\label{Lemma 2.12}
Let $a$ and $b$ be positive integers such that $a|b$. Let $S$
be a gcd-closed set and $x, y, z\in S$ with $G_S(x)=\{y\}$.

{\rm (i).} {\rm\cite[Lemma 2.5]{[Z-IJNT22]}} The integer $x^a-y^a$ divides each of $(x, z)^b-(y, z)^b$
and $[x,z]^b-[y,z]^b$.

{\rm (ii).} {\rm\cite[Lemma 2.8]{[ZL-BAMS22]}} If $r\in S$ and $r|x$, then $y^a[z,x]^b-x^a[z,y]^b$
is divisible by $r^a(y^a-x^a)$.
\end{lemma}

\begin{lemma}{\rm\cite{[WZ]}}\label{Lemma 2.13}
Let $a$ and $b$ be positive integers with $a|b$ and let $S$ be a gcd-closed set.
For $x,z\in S$ with $|G_S(x)|=2$, let
$G_S(x)=\{y_1, y_2\}$ and $y_3:=(y_1, y_2)$.
Assume that the set $\{u\in S: (x,z)|u|x\}$ satisfies
the condition $\mathcal G$.
Then each of the following is true:

{\rm (i).} The integer $x^a+y_3^a-y_1^a-y_2^a$ divides
each of $(z,x)^b+(z,y_3)^b-(z,y_1)^b-(z,y_2)^b$ and
$[z,x]^b+[z,y_3]^b-[z,y_1]^b-[z,y_2]^b$.

{\rm (ii).} For any $r\in S$ with $r|x$,
$x^a[z,y_3]^b+y_3^a[z,x]^b-y_1^a[z,y_2]^b-y_2^a[z,y_1]^b$
is divisible by $r^a(x^a+y_3^a-y_1^a-y_2^a)$.
\end{lemma}

\begin{lemma}\label{Lemma 2.14}
Let $a|b$ and let $S$ be a gcd-closed set satisfying
the condition $\mathcal G$. Let $x_m\in S$ with $G_S(x_m)=\{x_{m_1},x_{m_2},x_{m_3}\}$.
Define $x_{m_4}:=(x_{m_1},x_{m_2},x_{m_3})$ and
$x_{m_{ij}}:=(x_{m_i},x_{m_j})$ for $1\le i<j\le 3$.

{\rm (i).} The integer $x_m^a-x_{m_1}^a-x_{m_2}^a-x_{m_3}^a+x_{m_{12}}^a
+x_{m_{13}}^a+x_{m_{23}}^a-x_{m_4}^a$
divides $x_m^b-x_{m_1}^b-x_{m_2}^b-x_{m_3}^b+x_{m_{12}}^b
+x_{m_{13}}^b+x_{m_{23}}^b-x_{m_4}^b$.

{\rm (ii).} For any $x_t\in S$ with $x_t|x_m$, the fraction
$$
\frac{x_m^{b-a}-x_{m_1}^{b-a}-x_{m_2}^{b-a}-x_{m_3}^{b-a}+x_{m_{12}}^{b-a}
+x_{m_{13}}^{b-a}+x_{m_{23}}^{b-a}-x_{m_4}^{b-a}}
{x_t^a(x_m^{-a}-x_{m_1}^{-a}-x_{m_2}^{-a}-x_{m_3}^{-a}+x_{m_{12}}^{-a}
+x_{m_{13}}^{-a}+x_{m_{23}}^{-a}-x_{m_4}^{-a})}
$$
is an integer.
\end{lemma}

\begin{proof}
Since $S$ satisfies the condition $\mathcal G$,
we obtain that
$[x_{m_i}, x_{m_j}]=x_m$ for $1\le i<j\le 3$,
$\{x_{m_{12}}, x_{m_{13}}\}\subseteq G_S(x_{m_1})$,
$\{x_{m_{13}}, x_{m_{23}}\}\subseteq G_S(x_{m_3})$
and $\{x_{m_{12}}, x_{m_{23}}\}\subseteq G_S(x_{m_2})$.
Further, we can get that $[x_{m_{12}}, x_{m_{13}}]=x_{m_1}$,
$[x_{m_{13}}, x_{m_{23}}]=x_{m_3}$
and $[x_{m_{12}}, x_{m_{23}}]=x_{m_2}$.
This implies that
$$
\frac{x_m}{x_{m_2}}=\frac{x_{m_1}}{x_{m_{12}}}
=\frac{x_{m_{13}}}{x_{m_4}}=\frac{x_{m_3}}{x_{m_{23}}}
$$
and
$$
\frac{x_m}{x_{m_1}}=\frac{x_{m_2}}{x_{m_{12}}}
=\frac{x_{m_{23}}}{x_{m_4}}.
$$
We can derive that
\begin{align}
&x_{m_{12}}=\frac{x_{m_1}x_{m_2}}{x_m},\label{eq2.8}\\
&x_{m_{13}}=x_{m_4}\cdot \frac{x_m}{x_{m_2}},\label{eq2.9}\\
&x_{m_{23}}=x_{m_4}\cdot \frac{x_m}{x_{m_1}},\label{eq2.10}\\
&x_{m_3}=x_{m_{23}}\cdot\frac{x_m}{x_{m_2}}=x_{m_4}\cdot \frac{x_m}{x_{m_1}}\cdot\frac{x_m}{x_{m_2}}.\label{eq2.11}
\end{align}
It readily follows from \eqref{eq2.8} to \eqref{eq2.11} that
\begin{align*}
&x_m^a-x_{m_1}^a-x_{m_2}^a-x_{m_3}^a+x_{m_{12}}^a+x_{m_{13}}^a+x_{m_{23}}^a-x_{m_4}^a\\
=&x_m^a-x_{m_1}^a-x_{m_2}^a-\big(\frac{x_{m_4}x_m^2}{x_{m_1}x_{m_2}}\big)^a
+\big(\frac{x_{m_1}x_{m_2}}{x_m}\big)^a+\big(\frac{x_mx_{m_4}}{x_{m_2}}\big)^a
+\big(\frac{x_mx_{m_4}}{x_{m_1}}\big)^a-x_{m_4}^a\\
=&x_m^a-x_{m_1}^a-x_{m_2}^a+\big(\frac{x_{m_1}x_{m_2}}{x_m}\big)^a
+x_{m_4}^a\Big(\big(\frac{x_m}{x_{m_1}}\big)^a+\big(\frac{x_m}{x_{m_2}}\big)^a
-\big(\frac{x_m^2}{x_{m_1}x_{m_2}}\big)^a-1\Big)\\
=&\frac{x_m^{2a}-x_m^ax_{m_1}^a-x_m^ax_{m_2}^a+x_{m_1}^ax_{m_2}^a}{x_m^a}
-\frac{x_m^{2a}-x_m^ax_{m_1}^a-x_m^ax_{m_2}^a+x_{m_1}^ax_{m_2}^a}{x_{m_1}^ax_{m_2}^a}\cdot x_{m_4}^a\\
=&(x_m^{2a}-x_m^ax_{m_1}^a-x_m^ax_{m_2}^a+x_{m_1}^ax_{m_2}^a)\Big(\frac{1}{x_m^a}-\frac{x_{m_4}^a}{x_{m_1}^ax_{m_2}^a}\Big)\\
=&(x_m^a-x_{m_2}^a)(x_m^a-x_{m_1}^a)\frac{x_{m_1}^ax_{m_2}^a-x_m^ax_{m_4}^a}{x_m^ax_{m_1}^ax_{m_2}^a}\\
=&\Big(\Big(\frac{x_m}{x_{m_2}}\Big)^a-1\Big)\Big(\Big(\frac{x_m}{x_{m_1}}\Big)^a-1\Big)
\Big(\Big(\frac{x_{m_1}x_{m_2}}{x_m}\Big)^a-x_{m_4}^a\Big).
\end{align*}
Therefore
\begin{align*}
&\frac{x_m^b-x_{m_1}^b-x_{m_2}^b-x_{m_3}^b+x_{m_{12}}^b+x_{m_{13}}^b+x_{m_{23}}^b-x_{m_4}^b}
{x_m^a-x_{m_1}^a-x_{m_2}^a-x_{m_3}^a+x_{m_{12}}^a+x_{m_{13}}^a+x_{m_{23}}^a-x_{m_4}^a}\\
=&\frac{\big(\frac{x_m}{x_{m_2}}\big)^b-1}{\big(\frac{x_m}{x_{m_2}}\big)^a-1}\cdot
\frac{\big(\frac{x_m}{x_{m_1}}\big)^b-1}{\big(\frac{x_m}{x_{m_1}}\big)^a-1}
\cdot\frac{\big(\frac{x_{m_1}x_{m_2}}{x_m}\big)^b-x_{m_4}^b}
{\big(\frac{x_{m_1}x_{m_2}}{x_m}\big)^a-x_{m_4}^a}\in\mathbb Z.
\end{align*}
Part (i) is proved.

On the other hand, by \eqref{eq2.8} to \eqref{eq2.11}, we have
\begin{small}
\begin{align*}
&x_m^{-a}-x_{m_1}^{-a}-x_{m_2}^{-a}-x_{m_3}^{-a}+x_{m_{12}}^{-a}+x_{m_{13}}^{-a}
+x_{m_{23}}^{-a}-x_{m_4}^{-a}\\
=&\frac{1}{x_m^a}-\frac{1}{x_{m_1}^a}-\frac{1}{x_{m_2}^a}
-\big(\frac{x_{m_1}x_{m_2}}{x_{m_4}x_m^2}\big)^a
+\big(\frac{x_m}{x_{m_1}x_{m_2}}\big)^a
+\big(\frac{x_{m_2}}{x_{m_4}x_m}\big)^a+\big(\frac{x_{m_1}}
{x_{m_4}x_m}\big)^a-\frac{1}{x_{m_4}^a}\\
=&\frac{1}{x_m^a}-\frac{1}{x_{m_1}^a}-\frac{1}{x_{m_2}^a}
+\big(\frac{x_m}{x_{m_1}x_{m_2}}\big)^a
-\Big(\frac{x_{m_1}^ax_{m_2}^a}{x_m^{2a}}-\frac{x_{m_2}^a}
{x_m^a}-\frac{x_{m_1}^a}{x_m^a}+1\Big)
\cdot \frac{1}{x_{m_4}^a}\\
=&\frac{x_{m_1}^ax_{m_2}^a-x_m^ax_{m_2}^a-x_m^ax_{m_1}^a
+x_m^{2a}}{x_m^ax_{m_1}^ax_{m_2}^a}
-\Big(\frac{x_{m_1}^ax_{m_2}^a-x_m^ax_{m_2}^a-x_m^ax_{m_1}^a
+x_m^{2a}}{x_m^{2a}}\Big)\cdot \frac{1}{x_{m_4}^a}\\
=&(x_{m_1}^ax_{m_2}^a-x_m^ax_{m_2}^a-x_m^ax_{m_1}^a+x_m^{2a})
\Big(\frac{1}{x_m^ax_{m_1}^ax_{m_2}^a}-\frac{1}{x_m^{2a}x_{m_4}^a}\Big)\\
=&(x_m^a-x_{m_2}^a)(x_m^a-x_{m_1}^a)\frac{x_m^ax_{m_4}^a-x_{m_1}^ax_{m_2}^a}
{x_m^{2a}x_{m_1}^ax_{m_2}^ax_{m_4}^a}\\
=&\Big(\big(\frac{x_m}{x_{m_2}}\big)^a-1\Big)
\Big(\Big(\frac{x_m}{x_{m_1}}\Big)^a-1\Big)
\Big(x_{m_4}^a-\Big(\frac{x_{m_1}x_{m_2}}{x_m}\Big)^a\Big)
\frac{1}{x_m^ax_{m_4}^a}.
\end{align*}\end{small}
Hence
\begin{small}
\begin{align*}
&\frac{x_m^{b-a}-x_{m_1}^{b-a}-x_{m_2}^{b-a}-x_{m_3}^{b-a}+x_{m_{12}}^{b-a}
+x_{m_{13}}^{b-a}+x_{m_{23}}^{b-a}-x_{m_4}^{b-a}}
{x_t^a(x_m^{-a}-x_{m_1}^{-a}-x_{m_2}^{-a}-x_{m_3}^{-a}+x_{m_{12}}^{-a}
+x_{m_{13}}^{-a}+x_{m_{23}}^{-a}-x_{m_4}^{-a})}\\
=&-\frac{\big(\frac{x_m}{x_{m_2}}\big)^{b-a}-1}
{\big(\frac{x_m}{x_{m_2}}\big)^a-1}\cdot
\frac{\big(\frac{x_m}{x_{m_1}}\big)^{b-a}-1}
{\big(\frac{x_m}{x_{m_1}}\big)^a-1}
\cdot \frac{\big(\frac{x_{m_1}x_{m_2}}
{x_m}\big)^{b-a}-x_{m_4}^{b-a}}
{\big(\frac{x_{m_1}x_{m_2}}{x_m}\big)^a-x_{m_4}^a}
\cdot\Big(\frac{x_m}{x_t}\Big)^ax_{m_4}^a\in\mathbb Z.
\end{align*}
\end{small}
This finishes the proof of part (ii) and that of Lemma
\ref{Lemma 2.14}.
\end{proof}

The following fact is taken from the proof of Theorem 1.3
in \cite{[ZCH-JCTA2022]}.

\begin{lemma} \label{Lemma 2.16}
Let $S$ be a gcd-closed set satisfying the condition
$\mathcal{G}$ with $x_l, x_m\in S$. Let
$G_S(x_m):=\{x_{m_1},x_{m_2},x_{m_3}\}, x_{m_{ij}}
:=(x_{m_i},x_{m_j})\ \hbox{for}\ 1\le i<j\le 3$
and $x_{m_4}:=(x_{m_1},x_{m_2},x_{m_3})$. If
$(x_l, x_m)|x_{m_1}$, $(x_l, x_m)\ne x_{m_1}$, $(x_l, x_m)
\nmid x_{m_2}$ and $(x_l, x_m)\nmid x_{m_3}$, then
\begin{align}
&(x_l,x_{m_3})=(x_l, x_{m_{13}}),\label{eq2.12}\\
&[x_l,x_m]=[x_l,x_{m_2}],\label{eq2.13}\\
&[x_l,x_{m_1}]=[x_l,x_{m_{12}}],\label{eq2.14}\\
&[x_l, x_{m_3}]=[x_l, x_{m_{23}}],\label{eq2.15}\\
&[x_l, x_{m_{13}}]=[x_l, x_{m_4}].\label{eq2.16}
\end{align}

\end{lemma}

\section{Proof of Theorem \ref{Theorem 1.2}}

In the present section, we use the lemmas presented in Section 2
to show Theorem \ref{Theorem 1.2}. For arbitrary integers $l$ and $m$
with $1\le l, m\le n$, the functions $f(l,m)$ and $g(l,m)$ are defined by
\begin{align}
f(l,m):=\frac{1}{\alpha_{\xi_a}(x_m)}\sum\limits_{x_r|x_m}c_{rm}(x_l,x_r)^b\label{eq3.1}
\end{align}
and
\begin{align}
g(l,m):=\frac{1}{\alpha_{\xi_a}(x_m)}\sum\limits_{x_r|x_m}c_{rm}[x_l,x_r]^b,\label{eq3.2}
\end{align}
respectively.

\begin{lemma}\label{Lemma 2.17}
Let $S$ be a gcd-closed set satisfying the condition
$\mathcal G$ and let $x_l, x_m\in S$ with $|G_S(x_m)|=3$. Let
$a$ and $b$ be positive integers with $a|b$.
Then $f(l,m)\in \mathbb Z$ and $g(l,m)\in \mathbb Z$.
\end{lemma}

\begin{proof}
Let
$G_S(x_m)=\{x_{m_1},x_{m_2},x_{m_3}\},
x_{m_{ij}}=(x_{m_i},x_{m_j})\ \hbox{for}\ 1\le i<j\le 3,
x_{m_4}=(x_{m_1},x_{m_2},x_{m_3}).$
By Lemma \ref{Lemma 2.5}, one has
\begin{align}\label{eq3.4}
\alpha_{\xi_a}(x_m)=x_m^a-x_{m_1}^a-x_{m_2}^a-x_{m_3}^a
+x_{m_{12}}^a+x_{m_{13}}^a+x_{m_{23}}^a-x_{m_4}^a.
\end{align}
Since $x_m$ satisfies the condition $\mathcal G$,
then for any integers $i$ and $j$
with $1\le i<j\le 3$, we have $[x_{m_i},x_{m_j}]=x_m$.
Now we claim that
\begin{align*}
x_{m_4}\ne x_{m_{ij}}\ \hbox{for}\ 1\le i<j\le 3.
\end{align*}
Assume that $x_{m_4}=x_{m_{ij}}$. Without loss of generality,
we may let $x_{m_{12}}=x_{m_4}$. It follows that
\begin{align*}
x_m=&[x_{m_1},x_{m_2},x_{m_3}]\\
=&\frac{x_{m_1}x_{m_2}x_{m_3}(x_{m_1},x_{m_2},x_{m_3})}
{(x_{m_1},x_{m_2})(x_{m_1},x_{m_3})(x_{m_2},x_{m_3})}\\
=&\frac{x_{m_1}x_{m_2}x_{m_3}x_{m_4}}{x_{m_{12}}x_{m_{13}}x_{m_{23}}}
=\frac{x_{m_1}x_{m_2}x_{m_3}}{x_{m_{13}}x_{m_{23}}}\\
=&\frac{x_{m_2}}{x_{m_{23}}}[x_{m_1},x_{m_3}]
=\frac{x_{m_2}x_m}{x_{m_{23}}}.
\end{align*}
Thus $x_{m_2}=x_{m_{23}}$. This yields $x_{m_2}|x_{m_3}$.
It contradicts with the facts that
$G_S(x_m)=\{x_{m_1},x_{m_2},x_{m_3}\}$ and $x_{m_2}\ne x_{m_3}$.
The claim is proved. From the claim and Lemma \ref{Lemma 2.7}, we have
\begin{align}
\sum\limits_{x_r|x_m}c_{rm}(x_l,x_r)^b=&(x_l,x_m)^b
-(x_l,x_{m_1})^b-(x_l,x_{m_2})^b-(x_l,x_{m_3})^b\notag\\
&+(x_l,x_{m_{12}})^b+(x_l,x_{m_{13}})^b+(x_l,x_{m_{23}})^b
-(x_l,x_{m_4})^b,\label{eq3.6}\\
\sum\limits_{x_r|x_m}c_{rm}[x_l,x_r]^b=&[x_l,x_m]^b-[x_l,
x_{m_1}]^b-[x_l,x_{m_2}]^b-[x_l,x_{m_3}]^b\notag\\
&+[x_l,x_{m_{12}}]^b+[x_l,x_{m_{13}}]^b+[x_l,x_{m_{23}}]^b
-[x_l,x_{m_4}]^b.\label{eq3.7}
\end{align}
We only need to consider the following four cases
corresponding to the four divisibility diagrams in Fig. 1.

$$
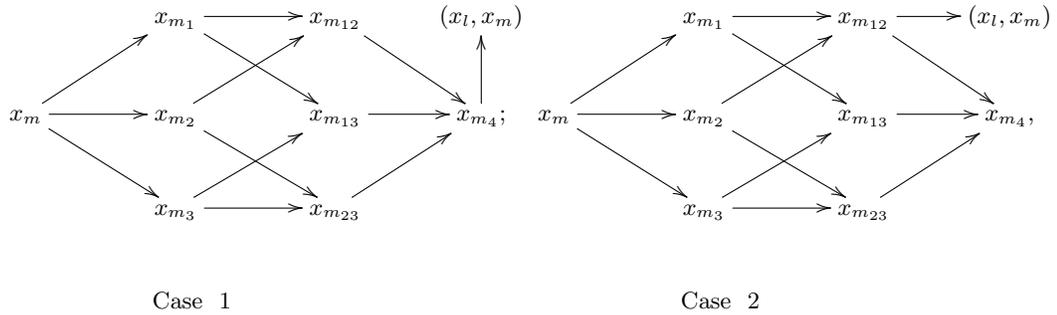
\hbox{{\bf Fig. 1.} Divisibility relations among}\
x_m,x_l,x_{m_i}(1\le i\le 4)\ \hbox{and}\ x_{m_{ij}}
(1\le i,j\le 3).
$$

$$\small{\xymatrix{
  &x_{m_1} \ar[r] \ar[dr] & x_{m_{12}}\ar[dr]&(x_l,x_m)\\
  x_{m}  \ar[dr] \ar[r] \ar[ur]& x_{m_2} \ar[ur] \ar[dr]&x_{m_{13}}\ar[r]&
  \ar[u]x_{m_4};\\
  &x_{m_3}\ar[r]\ar[ur]&x_{m_{23}}\ar[ur]& \\
  & \text{\rm \ \ \ \  Case \ 1}
  }
\xymatrix{
  &x_{m_1} \ar[r]\ar[dr] & x_{m_{12}}\ar[dr]\ar[r]&(x_l,x_m) \\
  x_{m}  \ar[dr] \ar[r] \ar[ur]& x_{m_2} \ar[ur] \ar[dr]&x_{m_{13}}\ar[r]&x_{m_4},\\
  &x_{m_3}\ar[r]\ar[ur]&x_{m_{23}}\ar[ur]&\\
  & \text{\rm \ \ \ \  Case \ 2}
   }}
$$
$$\small{\xymatrix{
  (x_l,x_m)&x_{m_1}\ar[l] \ar[r]\ar[dr] & x_{m_{12}}\ar[dr]&\\
  x_{m}  \ar[dr] \ar[r] \ar[ur]& x_{m_2} \ar[ur] \ar[dr]&x_{m_{13}}\ar[r]&
  x_{m_4},\\
  &x_{m_3}\ar[r]\ar[ur]&x_{m_{23}}\ar[ur]& \\
  & \text{\rm \ \ \ \  Case \ 3}
  }
\xymatrix{
 x_l\ar[d] &x_{m_1} \ar[r]\ar[dr] & x_{m_{12}}\ar[dr]& \\
  x_{m}  \ar[dr] \ar[r] \ar[ur]& x_{m_2} \ar[ur] \ar[dr]&x_{m_{13}}\ar[r]&x_{m_4},\\
  &x_{m_3}\ar[r]\ar[ur]&x_{m_{23}}\ar[ur]&\\
  & \text{\rm \ \ \ \  Case \ 4}
   }}
$$

{\sc Case 1}. $(x_l, x_m)|x_{m_4}$. In this case,
for any $i$ and $j$ with $1\le i<j\le 3$, we have
$$
(x_l, x_m)|(x_l, x_{m_4})|(x_l, x_{m_{ij}})|(x_l, x_{m_i})|(x_l, x_m).
$$
So
\begin{align}\label{eq3.8}
(x_l, x_m)=(x_l, x_{m_4})=(x_l, x_{m_{ij}})=(x_l, x_{m_i}).
\end{align}
The identities (\ref{eq3.6}) and (\ref{eq3.8}) yield that $\sum\limits_{x_r|x_m}c_{rm}(x_l,x_r)^b=0$.
Hence by (\ref{eq3.1}), we have $f(l,m)=0$.

On the other hand, from \eqref{eq3.7} and (\ref{eq3.8}), one has
\begin{align}
&\sum\limits_{x_r|x_m}c_{rm}[x_l,x_r]^b\notag\\
=&\frac{x_l^bx_m^b}{(x_l,x_m)^b}-\frac{x_l^bx_{m_1}^b}{(x_l,x_{m_1})^b}
-\frac{x_l^bx_{m_2}^b}{(x_l,x_{m_2})^b}-\frac{x_l^bx_{m_3}^b}{(x_l,x_{m_3})^b}\notag\\
&+\frac{x_l^bx_{m_{12}}^b}{(x_l,x_{m_{12}})^b}+\frac{x_l^bx_{m_{13}}^b}{(x_l,x_{m_{13}})^b}
+\frac{x_l^bx_{m_{23}}^b}{(x_l,x_{m_{23}})^b}-\frac{x_l^bx_{m_4}^b}{(x_l,x_{m_4})^b}\notag\\
=&\frac{x_l^b}{(x_l,x_m)^b}(x_m^b-x_{m_1}^b-x_{m_2}^b-x_{m_3}^b+x_{m_{12}}^b+x_{m_{13}}^b+x_{m_{23}}^b-x_{m_4}^b).\label{eq3.9}
\end{align}
It follows from (\ref{eq3.2}), \eqref{eq3.4}, \eqref{eq3.7}, (\ref{eq3.9}) and Lemma \ref{Lemma 2.14} (i) that
$$
g(l,m)=\frac{x_l^b}{(x_l,x_m)^b}
\frac{x_m^b-x_{m_1}^b-x_{m_2}^b-x_{m_3}^b+x_{m_{12}}^b
+x_{m_{13}}^b+x_{m_{23}}^b-x_{m_4}^b}
{x_m^a-x_{m_1}^a-x_{m_2}^a-x_{m_3}^a+x_{m_{12}}^a
+x_{m_{13}}^a+x_{m_{23}}^a-x_{m_4}^a}\in \mathbb Z
$$
as required. Case 1 is proved.

{\sc Case 2}. $(x_l, x_m)|x_{m_{12}}$,
$(x_l, x_m)\nmid x_{m_{13}}$ and $(x_l, x_m)\nmid x_{m_{23}}$.
Clearly, for any $i\in\{1,2\}$, we have
$(x_l, x_m)|(x_l, x_{m_{12}})|(x_l, x_{m_i})|(x_l, x_m)$.
Thus
\begin{align}\label{eq 3.10}
(x_l, x_m)=(x_l, x_{m_{12}})=(x_l, x_{m_i}) (i\in\{1,2\}).
\end{align}
Meanwhile, we have $(x_l, x_m)\nmid x_{m_3}$
and $(x_l, x_m)\nmid x_{m_4}$.
Since $A_S\big((x_l, x_m),x_{m_1}\big)(\subseteq S)$
satisfies the condition $\mathcal G$,
$x_{m_{13}}\in G_S(x_{m_1})$ and $(x_l, x_m)
\nmid x_{m_{13}}$, by Lemma \ref{Lemma 2.10} one has
$[(x_l, x_m), x_{m_{13}}]=x_{m_1}$. So
\begin{align}\label{eq 3.11}
[x_l, x_{m_1}]=[x_l, [(x_l, x_m), x_{m_{13}}]]=[x_l, x_{m_{13}}].
\end{align}
Similarly, we can prove that
\begin{align}
&[x_l, x_m]=[x_l, x_{m_3}],\label{eq 3.12}\\
&[x_l, x_{m_2}]=[x_l, x_{m_{23}}],\label{eq 3.13}\\
&[x_l, x_{m_{12}}]=[x_l, x_{m_4}].\label{eq 3.14}
\end{align}
By (\ref{eq3.6}) and \eqref{eq 3.10} to \eqref{eq 3.14}, we have
\begin{align}
&\sum\limits_{x_r|x_m}c_{rm}(x_l,x_r)^b\notag\\
=&(x_l,x_m)^b
-(x_l,x_{m_1})^b-(x_l,x_{m_2})^b-(x_l,x_{m_3})^b\notag\\
&+(x_l,x_{m_{12}})^b+(x_l,x_{m_{13}})^b+(x_l,x_{m_{23}})^b
-(x_l,x_{m_4})^b\notag\\
=&-(x_l,x_{m_3})^b+(x_l,x_{m_{13}})^b+(x_l,x_{m_{23}})^b-(x_l,x_{m_4})^b\notag\\
=&-\frac{x_l^bx_{m_3}^b}{[x_l,x_{m_3}]^b}+\frac{x_l^bx_{m_{13}}^b}{[x_l,x_{m_{13}}]^b}
+\frac{x_l^bx_{m_{23}}^b}{[x_l,x_{m_{23}}]^b}-\frac{x_l^bx_{m_4}^b}{[x_l,x_{m_4}]^b}\notag\\
=&-\frac{x_l^bx_{m_3}^b}{[x_l,x_m]^b}+\frac{x_l^bx_{m_{13}}^b}{[x_l,x_{m_1}]^b}
+\frac{x_l^bx_{m_{23}}^b}{[x_l,x_{m_2}]^b}-\frac{x_l^bx_{m_4}^b}{[x_l,x_{m_{12}}]^b}\notag\\
=&-\frac{x_l^bx_{m_3}^b(x_l,x_m)^b}{x_l^bx_m^b}+\frac{x_l^bx_{m_{13}}^b(x_l,x_{m_1})^b}{x_l^bx_{m_1}^b}
+\frac{x_l^bx_{m_{23}}^b(x_l,x_{m_2})^b}{x_l^bx_{m_2}^b}-\frac{x_l^bx_{m_4}^b(x_l,x_{m_{12}})^b}{x_l^bx_{m_{12}}^b}\notag\\
=&(x_l,x_m)^b\Big(-\big(\frac{x_{m_3}}{x_m}\big)^b+\big(\frac{x_{m_{13}}}{x_{m_1}}\big)^b
+\big(\frac{x_{m_{23}}}{x_{m_2}}\big)^b-\big(\frac{x_{m_4}}{x_{m_{12}}}\big)^b\Big).\label{eq 3.15}
\end{align}
Notice that both of $x_m$ and $x_{m_2}$ satisfy the condition $\mathcal G$,
$\{x_{m_1},x_{m_3}\}\subseteq G_S(x_m)$ and $\{x_{m_{12}},x_{m_{23}}\}\subseteq G_S(x_{m_2})$.
Then we have $[x_{m_1},x_{m_3}]=x_m$ and $[x_{m_{12}},x_{m_{23}}]=x_{m_2}$. So
\begin{align}\label{eq 3.16}
\frac{x_{m_3}}{x_m}=\frac{x_{m_{13}}}{x_{m_1}}
\end{align}
and
\begin{align}\label{eq 3.17}
\frac{x_{m_{23}}}{x_{m_2}}=\frac{x_{m_4}}{x_{m_{12}}}.
\end{align}
The identities (\ref{eq 3.15}) to (\ref{eq 3.17}) tell us that
\begin{align}
\sum\limits_{x_r|x_m}c_{rm}(x_l,x_r)^b=0.\label{eq 3.18}
\end{align}
It follows from (\ref{eq3.1}) and (\ref{eq 3.18}) that
$f(l,m)=0$.

On the other hand, it follows from (\ref{eq3.7}), and \eqref{eq 3.11} to \eqref{eq 3.14} that
\begin{align}
\sum\limits_{x_r|x_m}c_{rm}[x_l,x_r]^b=&[x_l,x_m]^b-[x_l,
x_{m_1}]^b-[x_l,x_{m_2}]^b-[x_l,x_{m_3}]^b\notag\\
&+[x_l,x_{m_{12}}]^b+[x_l,x_{m_{13}}]^b+[x_l,x_{m_{23}}]^b
-[x_l,x_{m_4}]^b\notag\\
=&0.\label{eq 3.19}
\end{align}
So by (\ref{eq3.2}) and (\ref{eq 3.19}), we have $g(l,m)=0$.
Case 2 is proved.

{\sc Case 3}. $(x_l, x_m)|x_{m_1}$, $(x_l, x_m)\nmid x_{m_2}$ and $(x_l, x_m)\nmid x_{m_3}$.
Since both of $x_{m_1}$ and $x_{m_3}$ satisfy the condition $\mathcal G$,
$\{x_{m_{12}}, x_{m_{13}}\}\subseteq G_S(x_{m_1})$ and
$\{x_{m_{13}}, x_{m_{23}}\}\subseteq G_S(x_{m_3})$,
we have $[x_{m_{12}}, x_{m_{13}}]=x_{m_1}$ and
$[x_{m_{13}}, x_{m_{23}}]=x_{m_3}$. So
\begin{align}
\frac{x_{m_1}}{x_{m_{12}}}=\frac{x_{m_{13}}}{x_{m_4}}
=\frac{x_{m_3}}{x_{m_{23}}}.\label{eq 3.20}
\end{align}
Consider the following two subcases:

{\sc Subcases 3-1}. $(x_l, x_m)\ne x_{m_1}$. Obviously,
$(x_l, x_m)|(x_l, x_{m_1})|(x_l, x_m)$.
Thus
\begin{align}
(x_l, x_m)=(x_l, x_{m_1}).\label{eq 3.21}
\end{align}
Moreover, it follows from \eqref{eq2.13} and \eqref{eq2.14} that
\begin{align}\label{eq 3.22}
(x_l, x_{m_2})=\dfrac{x_lx_{m_2}}{[x_l, x_{m_2}]}=\frac{x_lx_{m_2}}{[x_l, x_m]}
=\dfrac{x_lx_{m_2}(x_l, x_m)}{x_lx_m}=\frac{x_{m_2}(x_l, x_m)}{x_m}
\end{align}
and
\begin{align}\label{eq 3.23}
(x_l, x_{m_{12}})=\frac{x_lx_{m_{12}}}{[x_l, x_{m_{12}}]}=\frac{x_lx_{m_{12}}}{[x_l, x_{m_1}]}
=\frac{x_lx_{m_{12}}(x_l, x_{m_1})}{x_lx_{m_1}}=\frac{x_{m_{12}}(x_l, x_{m_1})}{x_{m_1}}.
\end{align}
Note that $\dfrac{x_{m_2}}{x_m}=\dfrac{x_{m_{12}}}{x_{m_1}}$.  By (\ref{eq 3.21}) to \eqref{eq 3.23}, one has
\begin{align}\label{eq 3.24}
(x_l, x_{m_2})=(x_l, x_{m_{12}}).
\end{align}
On the other hand, from (\ref{eq 3.21}), we have
\begin{small}
\begin{align}\label{eq 3.25}
(x_l, x_{m_4})=(x_l, x_{m_1}, x_{m_2}, x_{m_3})
=(x_l, x_m, x_{m_2}, x_{m_3})=(x_l, x_m, x_{m_{23}})
=(x_l, x_{m_{23}}).
\end{align}
\end{small}
It follows from (\ref{eq3.6}), \eqref{eq2.12}, \eqref{eq 3.21},
\eqref{eq 3.24} and (\ref{eq 3.25}) that
$\sum_{x_r|x_m}c_{rm}(x_l,x_r)^b=0$. Thus $f(l,m)=0$.
By (\ref{eq3.7}), and \eqref{eq2.13} to \eqref{eq2.16},
we have $\sum_{x_r|x_m}c_{rm}[x_l,x_r]^b=0$
and so $g(l,m)=0$.

{\sc Subcases 3-2}. $(x_l, x_m)=x_{m_1}$. Then
$x_{m_1}|x_l$ and so $x_{m_{12}}|x_l$,
$x_{m_{13}}|x_l$ and $x_{m_4}|x_l$. It is clear that
\begin{align}
&(x_l,x_{m_1})=x_{m_1},\label{eq 3.26}\\
&(x_l,x_{m_{12}})=x_{m_{12}},\label{eq 3.27}\\
&(x_l,x_{m_{13}})=x_{m_{13}},\label{eq 3.28}\\
&(x_l,x_{m_4})=x_{m_4}\label{eq 3.29}
\end{align}
and
\begin{align}\label{eq 3.30}
[x_l,x_{m_1}]=[x_l,x_{m_{12}}]=[x_l,x_{m_{13}}]
=[x_l,x_{m_4}]=x_l.
\end{align}
Moreover, we can obtain that
\begin{align}
&(x_l, x_{m_3})=(x_l, x_m, x_{m_3})=(x_{m_1}, x_{m_3})=x_{m_{13}},\label{eq 3.31}\\
&(x_l, x_{m_{23}})=(x_l, x_m, x_{m_{23}})=(x_{m_1}, x_{m_{23}})=x_{m_4},\label{eq 3.32}\\
&(x_l,x_{m_2})=(x_l, x_m, x_{m_2})=(x_{m_1}, x_{m_2})=x_{m_{12}}.\label{eq 3.33}
\end{align}
So from \eqref{eq 3.26} to \eqref{eq 3.29} and from \eqref{eq 3.31} to \eqref{eq 3.33}, we obtain that
$$(x_l, x_m)=(x_l,x_{m_1}),$$
$$(x_l,x_{m_{12}})=(x_l,x_{m_2}),$$
$$(x_l,x_{m_{13}})=(x_l, x_{m_3}),$$
and
$$(x_l,x_{m_4})=(x_l, x_{m_{23}}).$$
These together with (\ref{eq3.6}) imply that
$$
\sum\limits_{x_r|x_m}c_{rm}(x_l,x_r)^b=0
$$
and so $f(l,m)=0$.

Since $x_m$ satisfies the condition $\mathcal G$ and
$G_S(x_m)=\{x_{m_1}, x_{m_2}, x_{m_3}\}$.
one has $[x_{m_1}, x_{m_2}]=x_m$. This together with
$x_{m_1}|x_l$ yields that
\begin{align}\label{eq 3.34}
[x_l,x_m]=[x_l,[x_{m_1}, x_{m_2}]]=[x_l,x_{m_2}].
\end{align}
Meanwhile, by \eqref{eq 3.20}, \eqref{eq 3.31} and \eqref{eq 3.32} we have
\begin{small}
\begin{align}\label{eq 3.35}
[x_l, x_{m_3}]=\frac{x_l x_{m_3}}{(x_l, x_{m_3})}
=\frac{x_l x_{m_3}}{x_{m_{13}}}
=\frac{x_l x_{m_{23}}}{x_{m_4}}=\frac{x_l x_{m_{23}}}{(x_l, x_{m_{23}})}=[x_l, x_{m_{23}}].
\end{align}
\end{small}
It then follows from (\ref{eq3.7}), (\ref{eq 3.30}), (\ref{eq 3.34})
and (\ref{eq 3.35}) that
$$\sum\limits_{x_r|x_m}c_{rm}[x_l,x_r]^b=0.$$
This yields $g(l,m)=0$ as desired. Case 3 is proved.

{\sc Case 4}. $x_m|x_l$. It is easy to check that
$(x_l, x_m)=x_m$, and for any integers $r,i$ and $j$
with $1\le r\le 4$ and $1\le i<j\le 3$, one has
$(x_l, x_{m_r})=x_{m_r}$, $(x_l, x_{m_{ij}})=x_{m_{ij}}$
and $[x_l, x_m]=[x_l, x_{m_r}]=[x_l, x_{m_{ij}}]=x_l.$
So by (\ref{eq3.6}) and (\ref{eq3.7}), one derives that
\begin{align}\label{eq 3.36}
\sum\limits_{x_r|x_m}c_{rm}(x_l,x_r)^b=
x_m^b-x_{m_1}^b-x_{m_2}^b-x_{m_3}^b+x_{m_{12}}^b
+x_{m_{13}}^b+x_{m_{23}}^b-x_{m_4}^b
\end{align}
and
\begin{align}\label{eq 3.37}
\sum\limits_{x_r|x_m}c_{rm}[x_l,x_r]^b=0.
\end{align}
Thus from Lemma \ref{Lemma 2.14} (i), (\ref{eq3.1}), \eqref{eq3.4} and (\ref{eq 3.36}),
it deduces that
$$
f(l,m)=\frac{x_m^b-x_{m_1}^b-x_{m_2}^b-x_{m_3}^b
+x_{m_{12}}^b+x_{m_{13}}^b+x_{m_{23}}^b-x_{m_4}^b}
{x_m^a-x_{m_1}^a-x_{m_2}^a-x_{m_3}^a+x_{m_{12}}^a
+x_{m_{13}}^a+x_{m_{23}}^a-x_{m_4}^a}\in\mathbb Z.
$$
By (\ref{eq3.2}) and (\ref{eq 3.37}), we have
$g(l,m)=0$. Case 4 is proved and so is Lemma \ref{Lemma 2.17}.
\end{proof}

{\it Proof of Theorem \ref{Theorem 1.2}.}
Let $S$ be a gcd-closed set satisfying the condition
$\mathcal G$ and $\max_{x\in S}|G_S(x)|=3$.
For any integers $l$ and $j$ with $1\le l,j\le n$,
from Lemma \ref{Lemma 2.2}, we have
\begin{align*}
\left((S^b)(S^a)^{-1}\right)_{lj}=&\sum_{r=1}^n(x_l, x_r)^b\sum\limits_{x_r|x_m\atop x_j|x_m}\frac{c_{rm}c_{jm}}{\alpha_{\xi_a}(x_m)}\\
=&\sum\limits_{x_j|x_m}c_{jm}\sum\limits_{x_r|x_m}\frac{c_{rm}(x_l, x_r)^b}{\alpha_{\xi_a}(x_m)}\\
=&\sum\limits_{x_j|x_m}c_{jm}\frac{1}{\alpha_{\xi_a}(x_m)}\sum\limits_{x_r|x_m}c_{rm}(x_l, x_r)^b\\
=&\sum\limits_{x_j|x_m}c_{jm}f(l,m)
\end{align*}
and
\begin{align*}
\left([S^b](S^a)^{-1}\right)_{lj}
=&\sum_{r=1}^n[x_l, x_r]^b\sum\limits_{x_r|x_m\atop x_j|x_m}\frac{c_{rm}c_{jm}}{\alpha_{\xi_a}(x_m)}\\
=&\sum\limits_{x_j|x_m}c_{jm}\sum\limits_{x_r|x_m}
\frac{c_{rm}[x_l, x_r]^b}{\alpha_{\xi_a}(x_m)}\\
=&\sum\limits_{x_j|x_m}c_{jm}\frac{1}{\alpha_{\xi_a}(x_m)}
\sum\limits_{x_r|x_m}c_{rm}[x_l, x_r]^b\\
=&\sum\limits_{x_j|x_m}c_{jm}g(l,m),
\end{align*}
where $f(l,m)$ and $g(l,m)$ are defined as in
(\ref{eq3.1}) and  (\ref{eq3.2}), respectively.
By the definition of $c_{jm}$, we know that $c_{jm}\in\mathbb Z$.
So in order to show $\left((S^b)(S^a)^{-1}\right)_{lj}\in\mathbb Z$
and $\left([S^b](S^a)^{-1}\right)_{lj}\in\mathbb Z$,
we only need to prove that $f(l,m)\in\mathbb Z$ and $g(l,m)\in\mathbb Z$
for all $1\le l,m\le n$. Consider the following three cases:

{\sc Case 1}. $|G_S(x_m)|=1$. Let $G_S(x_m)=\{x_{m_0}\}$.
By Lemmas \ref{Lemma 2.5}, \ref{Lemma 2.6} and \ref{Lemma 2.12} (i),
we can arrive at that
$$
f(l, m)=\frac{(x_l, x_m)^b-(x_l, x_{m_0})^b}{x_m^a-x_{m_0}^a}\in\mathbb Z
$$
and
$$
g(l, m)=\frac{[x_l, x_m]^b-[x_l, x_{m_0}]^b}{x_m^a-x_{m_0}^a}\in\mathbb Z.
$$

{\sc Case 2}. $|G_S(x_m)|=2$. Let $G_S(x_m)=\{x_{m_{01}}, x_{m_{02}}\}$ and
$(x_{m_{01}}, x_{m_{02}})=x_{m_{03}}$.
It follows from Lemmas \ref{Lemma 2.5}, \ref{Lemma 2.6} and \ref{Lemma 2.13} (i) that
$$
f(l, m)=\frac{(x_l, x_m)^b-(x_l, x_{m_{01}})^b-(x_l, x_{m_{02}})^b+(x_l, x_{m_{03}})^b}
{x_m^a-x_{m_{01}}^a-x_{m_{02}}^a+x_{m_{03}}^a}\in\mathbb Z
$$
and
$$
g(l, m)=\frac{[x_l, x_m]^b-[x_l, x_{m_{01}}]^b-[x_l, x_{m_{02}}]^b+[x_l, x_{m_{03}}]^b}
{x_m^a-x_{m_{01}}^a-x_{m_{02}}^a+x_{m_{03}}^a}\in\mathbb Z.
$$

{\sc Case 3}. $|G_S(x_m)|=3$. Then from Lemma \ref{Lemma 2.17}, we know that
both of $f(l,m)$ and $g(l,m)$ are integers.

This completes the proof of Theorem \ref{Theorem 1.2}.
\qed

\section{Proof of Theorem \ref{Theorem 1.3}}

In this section, with help of the lemmas presented in Section 2,
we show Theorem \ref{Theorem 1.3}. For this purpose, for arbitrary
integers $l$, $m$ and $s$ with $1\le l, m, s\le n$ and $x_s|x_m$,
we define the function $h(l,m)$ as follows:
\begin{align}
h(l,m):=\frac{1}{x_s^a\alpha_{\frac{1}{\xi_a}}(x_m)}\sum\limits_{x_r|x_m}
\frac{c_{rm}[x_l,x_r]^b}{x_r^a}.\label{eq3.3}
\end{align}
We have the following result.

\begin{lemma}\label{Lemma 2.18}
Let $S$ be a gcd-closed set satisfying the condition
$\mathcal G$ and let $x_l,x_m\in S$ with $|G_S(x_m)|=3$.
Let $a$ and $b$ be positive integers with $a|b$.
Then $h(l,m)\in \mathbb Z$.
\end{lemma}

\begin{proof}
Let
\begin{align*}
&G_S(x_m)=\{x_{m_1},x_{m_2},x_{m_3}\},\\
&x_{m_{ij}}=(x_{m_i},x_{m_j})\ \hbox{for}\ 1\le i<j\le 3,\\
&x_{m_4}=(x_{m_1},x_{m_2},x_{m_3}).
\end{align*}
By Lemma \ref{Lemma 2.5}, one has
\begin{align}\label{eq 2.40}
\alpha_{\frac{1}{\xi_a}}(x_m)=x_m^{-a}-x_{m_1}^{-a}-x_{m_2}^{-a}
-x_{m_3}^{-a}+x_{m_{12}}^{-a}+x_{m_{13}}^{-a}
+x_{m_{23}}^{-a}-x_{m_4}^{-a}.
\end{align}

From Lemma \ref{Lemma 2.7}, we have
\begin{align}
\sum\limits_{x_r|x_m}\frac{c_{rm}[x_l,x_r]^b}{x_r^a}
=&\frac{[x_l,x_m]^b}{x_m^a}-\frac{[x_l,x_{m_1}]^b}{x_{m_1}^a}
-\frac{[x_l,x_{m_2}]^b}{x_{m_2}^a}-\frac{[x_l,x_{m_3}]^b}{x_{m_3}^a}\notag\\
&+\frac{[x_l,x_{m_{12}}]^b}{x_{m_{12}}^a}+\frac{[x_l,x_{m_{13}}]^b}{x_{m_{13}}^a}
+\frac{[x_l,x_{m_{23}}]^b}{x_{m_{23}}^a}-\frac{[x_l,x_{m_4}]^b}{x_{m_4}^a}.\label{eq3.38}
\end{align}
We only need to consider the following four cases as
the proof of Lemma \ref{Lemma 2.17}.

{\sc Case 1}. $(x_l, x_m)|x_{m_4}$.
From (\ref{eq3.8}) and (\ref{eq3.38}), one has
\begin{align}
&\sum\limits_{x_r|x_m}\frac{c_{rm}[x_l,x_r]^b}{x_r^a}\notag\\
=&\frac{x_l^bx_m^{b-a}}{(x_l,x_m)^b}-\frac{x_l^bx_{m_1}^{b-a}}{(x_l,x_{m_1})^b}
-\frac{x_l^bx_{m_2}^{b-a}}{(x_l,x_{m_2})^b}-\frac{x_l^bx_{m_3}^{b-a}}{(x_l,x_{m_3})^b}\notag\\
&+\frac{x_l^bx_{m_{12}}^{b-a}}{(x_l,x_{m_{12}})^b}+\frac{x_l^bx_{m_{13}}^{b-a}}{(x_l,x_{m_{13}})^b}
+\frac{x_l^bx_{m_{23}}^{b-a}}{(x_l,x_{m_{23}})^b}-\frac{x_l^bx_{m_4}^{b-a}}{(x_l,x_{m_4})^b}\notag\\
=&\frac{x_l^b}{(x_l,x_m)^b}(x_m^{b-a}-x_{m_1}^{b-a}-x_{m_2}^{b-a}-x_{m_3}^{b-a}+x_{m_{12}}^{b-a}
+x_{m_{13}}^{b-a}+x_{m_{23}}^{b-a}-x_{m_4}^{b-a}).\label{eq 2.43}
\end{align}
It follows from (\ref{eq3.3}), \eqref{eq 2.40}, (\ref{eq 2.43}) and Lemma \ref{Lemma 2.14} (ii) that
$$
h(l,m)=\frac{x_l^b}{(x_l,x_m)^b}
\frac{x_m^{b-a}-x_{m_1}^{b-a}-x_{m_2}^{b-a}-x_{m_3}^{b-a}+x_{m_{12}}^{b-a}
+x_{m_{13}}^{b-a}+x_{m_{23}}^{b-a}-x_{m_4}^{b-a}}
{x_s^a(x_m^{-a}-x_{m_1}^{-a}-x_{m_2}^{-a}-x_{m_3}^{-a}+x_{m_{12}}^{-a}+x_{m_{13}}^{-a}+x_{m_{23}}^{-a}-x_{m_4}^{-a})}\in \mathbb Z
$$
as expected.

{\sc Case 2}. $(x_l, x_m)|x_{m_{12}}$, $(x_l, x_m)\nmid x_{m_{13}}$ and $(x_l, x_m)\nmid x_{m_{23}}$.
By (\ref{eq3.38}) and \eqref{eq 3.10} to \eqref{eq 3.14}, we have
\begin{small}
\begin{align}
&\sum\limits_{x_r|x_m}\frac{c_{rm}[x_l,x_r]^b}{x_r^a}\notag\\
=&\frac{[x_l,x_m]^b}{x_m^a}-\frac{[x_l,x_{m_1}]^b}{x_{m_1}^a}
-\frac{[x_l,x_{m_2}]^b}{x_{m_2}^a}-\frac{[x_l,x_m]^b}{x_{m_3}^a}\notag\\
&+\frac{[x_l,x_{m_{12}}]^b}{x_{m_{12}}^a}+\frac{[x_l,x_{m_1}]^b}{x_{m_{13}}^a}
+\frac{[x_l,x_{m_2}]^b}{x_{m_{23}}^a}-\frac{[x_l,x_{m_{12}}]^b}{x_{m_4}^a}\notag\\
=&[x_l,x_m]^b\big(\frac{1}{x_m^a}-\frac{1}{x_{m_3}^a}\big)+[x_l, x_{m_1}]^b
\big(\frac{1}{x_{m_{13}}^a}-\frac{1}{x_{m_1}^a}\big)
+[x_l, x_{m_2}]^b\big(\frac{1}{x_{m_{23}}^a}-\frac{1}{x_{m_2}^a}\big)\notag\\
&+[x_l, x_{m_{12}}]^b\big(\frac{1}{x_{m_{12}}^a}-\frac{1}{x_{m_4}^a}\big)\notag\\
=&\frac{x_l^bx_m^b}{(x_l,x_m)^b}\big(\frac{1}{x_m^a}-\frac{1}{x_{m_3}^a}\big)+\frac{x_l^bx_{m_1}^b}{(x_l, x_{m_1})^b}
\big(\frac{1}{x_{m_{13}}^a}-\frac{1}{x_{m_1}^a}\big)
+\frac{x_l^bx_{m_2}^b}{(x_l, x_{m_2})^b}\big(\frac{1}{x_{m_{23}}^a}-\frac{1}{x_{m_2}^a}\big)\notag\\
&+\frac{x_l^bx_{m_{12}}^b}{(x_l, x_{m_{12}})^b}\big(\frac{1}{x_{m_{12}}^a}-\frac{1}{x_{m_4}^a}\big)\notag\\
=&\frac{x_l^b}{(x_l,x_m)^b}\big(x_m^{b-a}-x_{m_1}^{b-a}-x_{m_2}^{b-a}+x_{m_{12}}^{b-a}
-\frac{x_m^b}{x_{m_3}^a}+\frac{x_{m_1}^b}{x_{m_{13}}^a}+\frac{x_{m_2}^b}{x_{m_{23}}^a}-\frac{x_{m_{12}}^b}{x_{m_4}^a}\big)
.\label{eq 3.40}
\end{align}
\end{small}
Since $S$ satisfies the condition $\mathcal G$, we have
$[x_{m_i}, x_{m_j}]=x_m$ for $1\le i<j\le 3$, $\{x_{m_{12}}, x_{m_{13}}\}\subseteq G_S(x_{m_1})$
and $\{x_{m_{13}}, x_{m_{23}}\}\subseteq G_S(x_{m_3})$.
This implies that
\begin{align}\label{eq 3.41}
\frac{x_{m_2}}{x_{m_{23}}}=\frac{x_m}{x_{m_3}}
=\frac{x_{m_1}}{x_{m_{13}}}=\frac{x_{m_{12}}}{x_{m_4}}:=\lambda
\end{align}
and
\begin{align}\label{eq 3.42}
\frac{x_m}{x_{m_1}}=\frac{x_{m_2}}{x_{m_{12}}}
=\frac{x_{m_3}}{x_{m_{13}}}=\frac{x_{m_{23}}}{x_{m_4}}:=\delta.
\end{align}
From (\ref{eq 3.40}) to (\ref{eq 3.42}), we obtain that
\begin{small}
\begin{align}
\sum\limits_{x_r|x_m}\frac{c_{rm}[x_l,x_r]^b}{x_r^a}=&
\frac{x_l^b}{(x_l,x_m)^b}\big(x_m^{b-a}-x_{m_1}^{b-a}-x_{m_2}^{b-a}+x_{m_{12}}^{b-a}
-\frac{x_m^b}{x_{m_3}^a}+\frac{x_{m_1}^b}{x_{m_{13}}^a}+\frac{x_{m_2}^b}{x_{m_{23}}^a}-\frac{x_{m_{12}}^b}{x_{m_4}^a}\big)\notag\\
=&\frac{x_l^b}{(x_l,x_m)^b}\big(x_m^{b-a}-x_{m_1}^{b-a}-x_{m_2}^{b-a}+x_{m_{12}}^{b-a}
-\lambda ^ax_m^{b-a}+\lambda ^ax_{m_1}^{b-a}+\lambda ^ax_{m_2}^{b-a}-\lambda ^ax_{m_{12}}^{b-a}\big)\notag\\
=&\frac{x_l^b}{(x_l,x_m)^b}(1-\lambda^a)\big(x_m^{b-a}-x_{m_1}^{b-a}-x_{m_2}^{b-a}+x_{m_{12}}^{b-a}\big)\notag\\
=&\frac{x_l^b}{(x_l,x_m)^b}(1-\lambda^a)\big(x_m^{b-a}-\big(\frac{x_m}{\delta}\big)^{b-a}-x_{m_2}^{b-a}+\big(\frac{x_{m_2}}{\delta}\big)^{b-a}\big)\notag\\
=&\frac{x_l^b}{(x_l,x_m)^b}(1-\lambda^a)(\delta^{b-a}-1)\Big(\big(\frac{x_m}{\delta}\big)^{b-a}-\big(\frac{x_{m_2}}{\delta}\big)^{b-a}\big).\label{eq 3.43}
\end{align}
\end{small}
By (\ref{eq 2.40}), (\ref{eq 3.41}) and (\ref{eq 3.42}), one has
\begin{align}\label{eq 3.44}
\alpha_{\frac{1}{\xi_a}}(x_m)=(1-\lambda^a)\delta^a(\delta^a-1)\frac{\big(\frac{x_m}{\delta}\big)^a-\big(\frac{x_{m_2}}{\delta}\big)^a}{x_m^ax_{m_2}^a}.
\end{align}
It follows from (\ref{eq 3.43}) and (\ref{eq 3.44}) that
\begin{align*}
h(l,m)=\frac{x_m^a}{x_s^a}\frac{x_{m_2}^a}{\delta^a}\frac{x_l^b}{(x_l,x_m)^b}\frac{\big(\frac{x_m}{\delta}\big)^{b-a}
-\big(\frac{x_{m_2}}{\delta}\big)^{b-a}}{\big(\frac{x_m}{\delta}\big)^a
-\big(\frac{x_{m_2}}{\delta}\big)^a}\frac{\delta^{b-a}-1}{\delta^a-1}\in\mathbb Z
\end{align*}
as desired.

{\sc Case 3}. $(x_l, x_m)|x_{m_1}$, $(x_l, x_m)\nmid x_{m_2}$ and $(x_l, x_m)\nmid x_{m_3}$.
We should consider the following two cases:

{\sc Subcases 3-1}. $(x_l, x_m)\ne x_{m_1}$. Obviously, $(x_l, x_m)|(x_l, x_{m_1})|(x_l, x_m)$.
Thus
\begin{align}
(x_l, x_m)=(x_l, x_{m_1}).\label{eq3.45}
\end{align}
Notice that $(x_l, x_m)\nmid x_{m_2}$, $(x_l, x_m)\nmid x_{m_3}$,
$(x_l, x_m)\nmid x_{m_{12}}$, $(x_l, x_m)\nmid x_{m_{13}}$, $\{x_{m_2},x_{m_3}\}\subseteq G_S(x_m)$ and
$\{x_{m_{12}},x_{m_{13}}\}\subseteq G_S(x_{m_1})$. Since $A_S((x_l,x_m),x_m)(\subseteq S)$ satisfies
the condition $\mathcal G$, from Lemma \ref{Lemma 2.10}, we have
\begin{align}\label{eq3.46}
[(x_l, x_m),x_{m_2}]=x_m=[(x_l, x_m),x_{m_3}]
\end{align}
and
\begin{align}\label{eq3.47}
[(x_l, x_m),x_{m_{12}}]=x_{m_1}=[(x_l, x_m),x_{m_{13}}].
\end{align}
So from \eqref{eq3.46} and \eqref{eq3.47}, we get that
\begin{align}\label{eq3.48}
[x_l, x_{m_2}]=[x_l,(x_l, x_m),x_{m_2}]=[x_l,x_m]=[x_l,(x_l, x_m),x_{m_3}]=[x_l, x_{m_3}]
\end{align}
and
\begin{align}\label{eq3.49}
[x_l, x_{m_{12}}]=[x_l,(x_l, x_m),x_{m_{12}}]=[x_l,x_{m_1}]=[x_l,(x_l, x_m),x_{m_{13}}]=[x_l, x_{m_{13}}].
\end{align}
Hence by \eqref{eq2.15}, \eqref{eq2.16} \eqref{eq3.48} and \eqref{eq3.49}, we have
\begin{align}
&[x_l, x_{m_2}]=[x_l, x_m]=[x_l, x_{m_3}]=[x_l, x_{m_{23}}],\label{eq3.50}\\
&[x_l, x_{m_{12}}]=[x_l, x_{m_1}]=[x_l, x_{m_{13}}]=[x_l, x_{m_4}].\label{eq3.51}
\end{align}
It then follows from (\ref{eq3.38}), \eqref{eq 3.42} , \eqref{eq3.45}, \eqref{eq3.50} and \eqref{eq3.51} that
\begin{small}
\begin{align}
&\sum\limits_{x_r|x_m}\frac{c_{rm}[x_l,x_r]^b}{x_r^a}\notag\\
=&\frac{[x_l, x_m]^b}{x_m^a}-\frac{[x_l, x_{m_1}]^b}{x_{m_1}^a}
-\frac{[x_l, x_m]^b}{x_{m_2}^a}-\frac{[x_l, x_m]^b}{x_{m_3}^a}\notag\\
&+\frac{[x_l, x_{m_1}]^b}{x_{m_{12}}^a}+\frac{[x_l, x_{m_1}]^b}{x_{m_{13}}^a}
+\frac{[x_l, x_m]^b}{x_{m_{23}}^a}-\frac{[x_l, x_{m_1}]^b}{x_{m_4}^a}\notag\\
=&[x_l, x_m]^b\big(x_m^{-a}-x_{m_2}^{-a}-x_{m_3}^{-a}+x_{m_{23}}^{-a}\big)
+[x_l, x_{m_1}]^b\big(-x_{m_1}^{-a}+x^{-a}_{m_{12}}+x^{-a}_{m_{13}}-x^{-a}_{m_4}\big)\notag\\
=&\frac{x_l^bx_m^b}{(x_l, x_m)^b}\big(x_m^{-a}-x_{m_2}^{-a}-x_{m_3}^{-a}+x_{m_{23}}^{-a}\big)
+\frac{x_l^bx_{m_1}^b}{(x_l, x_{m_1})^b}\big(-x_{m_1}^{-a}+x^{-a}_{m_{12}}+x^{-a}_{m_{13}}-x^{-a}_{m_4}\big)\notag\\
=&\frac{x_l^b}{(x_l, x_m)^b}\big(x_m^b(x_m^{-a}-x_{m_2}^{-a}-x_{m_3}^{-a}+x_{m_{23}}^{-a})
+x_{m_1}^b(-x_{m_1}^{-a}+x^{-a}_{m_{12}}+x^{-a}_{m_{13}}-x^{-a}_{m_4})\big)\notag\\
=&\frac{x_l^b}{(x_l, x_m)^b}\big(x_m^b(x_m^{-a}-x_{m_2}^{-a}-x_{m_3}^{-a}+x_{m_{23}}^{-a})
+x_{m_1}^b\delta^a(-x_m^{-a}+x^{-a}_{m_2}+x^{-a}_{m_3}-x^{-a}_{m_{23}})\big)\notag\\
=&\frac{x_l^b}{(x_l, x_m)^b}(x_m^b-x^b_{m_1}\delta^a)(x_m^{-a}-x_{m_2}^{-a}-x_{m_3}^{-a}+x_{m_{23}}^{-a})\notag\\
=&\frac{x_m^ax_l^b}{(x_l, x_m)^b}(x_m^{b-a}-x_{m_1}^{b-a})(x_m^{-a}-x_{m_2}^{-a}-x_{m_3}^{-a}+x_{m_{23}}^{-a}).
\label{eq3.52}
\end{align}
\end{small}
Using (\ref{eq 2.40}) and (\ref{eq 3.42})  gives us that
\begin{align}\label{eq3.53}
\alpha_{\frac{1}{\xi_a}}(x_m)=x^{-a}_{m_1}(x^a_{m_1}-x_m^a)(x_m^{-a}-x_{m_2}^{-a}-x_{m_3}^{-a}+x_{m_{23}}^{-a}).
\end{align}
It follows from (\ref{eq3.3}), (\ref{eq3.52}) and (\ref{eq3.53}) that
\begin{align*}
h(l,m)=x^a_{m_1}\big(\frac{x_m}{x_s}\big)^a\big(\frac{x_l}{(x_l,x_m)}\big)^b\frac{x_m^{b-a}-x_{m_1}^{b-a}}{x_{m_1}^a-x_m^a}\in\mathbb Z
\end{align*}
as asserted.

{\sc Subcases 3-2}. $(x_l, x_m)=x_{m_1}$. Then $x_{m_1}|x_l$ and so $x_{m_{12}}|x_l$,
$x_{m_{13}}|x_l$ and $x_{m_4}|x_l$. Since $x_m$ satisfies the condition $\mathcal G$,
we have $[x_{m_1}, x_{m_3}]=x_m$. It follows that
$$
\frac{x_m}{x_{m_1}}=\frac{x_{m_3}}{x_{m_{13}}}.
$$
Thus from \eqref{eq 3.31}, we get that
\begin{align}
[x_l,x_m]=\frac{x_lx_m}{(x_l,x_m)}=\frac{x_lx_m}{x_{m_1}}
=\frac{x_lx_{m_3}}{x_{m_{13}}}=\frac{x_lx_{m_3}}{(x_l,x_{m_3})}=[x_l,x_{m_3}].\label{eq3.54}
\end{align}
Therefore by \eqref{eq 3.34}, \eqref{eq 3.35} and \eqref{eq3.54}, one has
\begin{align}\label{eq3.55}
[x_l,x_{m_2}]=[x_l,x_m]=[x_l,x_{m_3}]=[x_l, x_{m_{23}}].
\end{align}
One can deduce from \eqref{eq 3.30}, (\ref{eq3.38}), \eqref{eq 3.42} and \eqref{eq3.55} that
\begin{align}
&\sum\limits_{x_r|x_m}\frac{c_{rm}[x_l,x_r]^b}{x_r^a}\notag\\
=&\frac{[x_l,x_m]^b}{x_m^a}-\frac{x_l^b}{x_{m_1}^a}-\frac{[x_l,x_m]^b}{x_{m_2}^a}-\frac{[x_l,x_m]^b}{x_{m_3}^a}
+\frac{x_l^b}{x_{m_{12}}^a}+\frac{x_l^b}{x_{m_{13}}^a}+\frac{[x_l,x_m]^b}{x_{m_{23}^a}}-\frac{x_l^b}{x_{m_4}^a}\notag\\
=&[x_l,x_m]^b(x_m^{-a}-x_{m_2}^{-a}-x_{m_3}^{-a}+x_{m_{23}}^{-a})
+x_l^b(x^{-a}_{m_{12}}+x_{m_{13}}^{-a}-x_{m_1}^{-a}-x_{m_4}^{-a}\big)\notag\\
=&\frac{x_l^bx_m^b}{(x_l,x_m)^b}\big(x_m^{-a}-x_{m_2}^{-a}-x_{m_3}^{-a}+x_{m_{23}}^{-a}\big)
+x_l^b\big(x^{-a}_{m_{12}}+x_{m_{13}}^{-a}-x_{m_1}^{-a}-x_{m_4}^{-a}\big)\notag\\
=&\frac{x_l^bx_m^b}{x_{m_1}^b}(x_m^{-a}-x_{m_2}^{-a}-x_{m_3}^{-a}+x_{m_{23}}^{-a})
+x_l^b\big(x^{-a}_{m_{12}}+x_{m_{13}}^{-a}-x_{m_1}^{-a}-x_{m_4}^{-a}\big)\notag\\
=&\frac{x_l^bx_m^b}{x_{m_1}^b}(x_m^{-a}-x_{m_2}^{-a}-x_{m_3}^{-a}+x_{m_{23}}^{-a})
-x_l^b\delta^a(x_m^{-a}-x_{m_2}^{-a}-x_{m_3}^{-a}+x_{m_{23}}^{-a})\notag\\
=&x_l^b\cdot\frac{x_m^b-x_m^ax_{m_1}^{b-a}}{x_{m_1}^b}
(x_m^{-a}-x_{m_2}^{-a}-x_{m_3}^{-a}+x_{m_{23}}^{-a}).\label{eq3.56}
\end{align}
Combining \eqref{eq3.53} and \eqref{eq3.56} gives us that
\begin{align*}
h(l,m)=x_{m_1}^a\Big(\frac{x_m}{x_s}\Big)^a\Big(\frac{x_l}{(x_l,x_m)}\Big)^b
\frac{x_m^{b-a}-x_{m_1}^{b-a}}{x_{m_1}^a-x_m^a}\in\mathbb Z
\end{align*}
as claimed.

{\sc Case 4}. $x_m|x_l$. It is easy to check that $(x_l, x_m)=x_m$, and for any integers
$r,i$ and $j$ with $1\le r\le 4$ and $1\le i<j\le 3$, one has $(x_l, x_{m_r})=x_{m_r}$,
$(x_l, x_{m_{ij}})=x_{m_{ij}}$ and
$$
[x_l, x_m]=[x_l, x_{m_r}]=[x_l, x_{m_{ij}}]=x_l.
$$
So by (\ref{eq3.38}), one derives that
\begin{align}\label{eq3.57}
\sum\limits_{x_r|x_m}\frac{c_{rm}[x_l,x_r]^b}{x_r^a}=x_l^b
(x_m^{-a}-x_{m_1}^{-a}-x_{m_2}^{-a}-x_{m_3}^{-a}
+x_{m_{12}}^{-a}+x_{m_{13}}^{-a}+x_{m_{23}}^{-a}-x_{m_4}^{-a}).
\end{align}
Notice that $x_s|x_m$ and $x_m|x_l$. Thus $x_s|x_l$.
From (\ref{eq3.3}), (\ref{eq 2.40}) and (\ref{eq3.57}),
it deduces that
$$
h(l,m)=\frac{x_l^b}{x_s^a}=x_l^{b-a}\Big(\frac{x_l}{x_s}\Big)
^a\in\mathbb Z
$$
as desired. This finishes the proof of Lemma \ref{Lemma 2.18}.
\end{proof}

Finally, we present the proof of Theorem \ref{Theorem 1.3}.

{\it Proof of Theorem \ref{Theorem 1.3}.} Let $S$ be a gcd-closed
set satisfying the condition $\mathcal G$ and $\max_{x\in S}|G_S(x)|=3$.
For arbitrary integers $l$ and $s$ with $1\le l,s\le n$,
from Lemma \ref{Lemma 2.9}, we have
\begin{align*}
\left([S^b][S^a]^{-1}\right)_{ls}=&\sum_{r=1}^n[x_l, x_r]^b\frac{1}{x_r^ax_s^a}
\sum\limits_{x_r|x_m\atop x_s|x_m}\frac{c_{rm}c_{sm}}{\alpha_{\frac{1}{\xi_a}}(x_m)}\\
=&\sum\limits_{x_s|x_m}c_{sm}\frac{1}{x_s^a\alpha_{\frac{1}{\xi_a}}(x_m)}\sum\limits_{x_r|x_m}\frac{c_{rm}[x_l, x_r]^b}{x_r^a}\\
=&\sum\limits_{x_s|x_m}c_{sm}h(l,m),
\end{align*}
where $h(l,m)$ is defined as in (\ref{eq3.3}).
To show the required result, it suffices to prove that for any positive integers $m$ such
that $x_s|x_m$, one has $h(l,m)\in\mathbb Z$ which will be done as follows.
Consider the following three cases:

{\sc Case 1}. $|G_S(x_m)|=1$. Let $G_S(x_m)=\{x_{m_0}\}$.
By Lemmas \ref{Lemma 2.5}, \ref{Lemma 2.6} and \ref{Lemma 2.12} (ii),
we can arrive at that
$$
h(l, m)=\frac{x^a_{m_0}[x_l, x_m]^b-x_m^a[x_l, x_{m_0}]^b}{x_s^a(x_{m_0}^a-x_m^a)}\in\mathbb Z.
$$

{\sc Case 2}. $|G_S(x_m)|=2$. Let $G_S(x_m)=\{x_{m_{01}}, x_{m_{02}}\}$ and
$(x_{m_{01}}, x_{m_{02}})=x_{m_{03}}$.
It follows from Lemmas \ref{Lemma 2.5}, \ref{Lemma 2.6} and
\ref{Lemma 2.13} (ii) that
\begin{align*}
h(l,m)
=&\frac{\frac{[x_l,x_m]^b}{x_m^a}-\frac{[x_l,x_{m_{01}}]^b}
{x_{m_{01}}^a}-\frac{[x_l,x_{m_2}]^b}{x_{m_{02}}^a}+
\frac{[x_l,x_{m_{03}}]^b}{x_{m_{03}}^a}}{x_s^a\big(\frac{1}{x_m^a}
-\frac{1}{x_{m_{01}}^a}-\frac{1}{x_{m_{02}}^a}+\frac{1}{x_{m_{03}}^a}\big)}\\
=&\frac{x_m^a[x_l,x_{m_{03}}]^b+x_{m_{03}}^a[x_l,x_m]^b-x_{m_{02}}^a
[x_l,x_{m_{01}}]^b-x_{m_{01}}^a[x_l,x_{m_{02}}]^b}
{x_s^a({x_m^a+x_{m_{03}}^a-x_{m_{01}}^a-x_{m_{02}}^a})}\in {\mathbb Z}.
\end{align*}

{\sc Case 3}. $|G_S(x_m)|=3$. Then by Lemma \ref{Lemma 2.18},
we know that $h(l,m)$ is an integer.

This concludes the proof of Theorem \ref{Theorem 1.3}.
\qed


\end{document}